\documentclass[10pt]{article}
\usepackage{amsmath,amssymb,amsthm}
\usepackage{color}
\usepackage[nohead,margin=1.0in]{geometry}
\usepackage{booktabs}
\usepackage{algorithm,algorithmic}
\usepackage{verbatim}
\usepackage{url}
\usepackage{graphicx}
\usepackage{epstopdf}
\graphicspath{{./pics/}}

 \usepackage{array}
\theoremstyle{plain}
\newtheorem{theorem}{Theorem}[section]
\newtheorem{remark}{Remark}[section]

\newtheorem{lemma}{Lemma}[section]

\newtheorem{proposition}{Proposition}[section]

\newtheorem{example}{Example}[section]

\numberwithin{equation}{section}

\renewcommand{\d}{\mathrm{d}}

\title{An Inverse Potential Problem for Subdiffusion: Stability and Reconstruction\thanks{The work of B. Jin is supported by UK EPSRC grant EP/T000864/1, and
the research of Z. Zhou is supported by Hong Kong RGC grant (No. 25300818).}}
\author{Bangti Jin\thanks{Department of Computer Science, University College London, Gower Street, London WC1E 6BT, UK (\texttt{b.jin@ucl.ac.uk,~bangti.jin@gmail.com})}
\and Zhi Zhou\thanks{Department of Applied Mathematics,
The Hong Kong Polytechnic University, Kowloon, Hong Kong (\texttt{zhizhou@polyu.edu.hk})}
}

\begin{document}

\maketitle
\begin{abstract}
In this work, we study the inverse problem of recovering a potential
coefficient in the subdiffusion model, which involves a Djrbashian-Caputo derivative of order
$\alpha\in(0,1)$ in time, from the terminal data. We prove that the inverse problem
is locally Lipschitz for small terminal time, under certain conditions on the initial data. This result extends the result in
\cite{ChoulliYamamoto:1997} for the standard parabolic case to the fractional case. The analysis relies on refined
properties of two-parameter Mittag-Leffler functions, e.g., complete monotonicity and asymptotics.
Further, we develop an efficient and easy-to-implement algorithm for numerically recovering the coefficient based on (preconditioned) fixed point iteration and Anderson
acceleration. The efficiency and accuracy of the algorithm is illustrated with several numerical
examples.
\end{abstract}

\section{Introduction}
Let $\Omega\subset\mathbb{R}^d$ ($d=1,2,3$) be a smooth open bounded domain with a boundary $\partial\Omega$.
Consider the following initial boundary value problem for subdiffusion:
\begin{equation}\label{eqn:fde}
  \left\{\begin{aligned}
    \partial_t^\alpha u &= \Delta u + q(x)u,  \quad \mbox{in }\Omega\times(0,T],\\
    u(\cdot,0) &= u_0, \quad \mbox{in }\Omega,\\
    u &=0, \quad \mbox{on }\partial\Omega\times(0,T],
  \end{aligned}\right.
\end{equation}
where $T>0$ is the final time and $u_0$ is the initial data.
The notation $\partial_t^\alpha u$ denotes the Djrbashian-Caputo derivative
of order $\alpha\in (0,1)$ (in time), defined by \cite[p. 91]{KilbasSrivastavaTrujillo:2006}
\begin{equation*}
  \partial_t^\alpha u(t) = \frac{1}{\Gamma(1-\alpha)}\int_0^t (t-s)^{-\alpha}u'(s)\d s,
\end{equation*}
where
\begin{equation*}
  \Gamma(z)=\int_0^{\infty}s^{z-1}e^{-s}\d s,\quad \mbox{for }\Re z>0,
\end{equation*}
denotes Euler's Gamma function. For smooth functions $u$, the fractional derivative $\partial_t^\alpha u$
recovers the usual first-order derivative $u'(t)$ as $\alpha\to1^-$. The function $q$ refers to the radiativity or reaction
coefficient or potential in the standard parabolic case, dependent of the specific applications. Throughout, we
denote by $u(q)$ the solution of problem \eqref{eqn:fde} that corresponds to a given potential $q\in L^2(\Omega)$.

The model \eqref{eqn:fde} is a direct extension of the standard subdiffusion model, which has a trivial 
potential $q$ (i.e., $q\equiv0$), and can faithfully describe anomalously slow diffusion processes. At a 
microscopical level, standard subdiffusion can be described by continuous time random walk, 
where the waiting time distribution between consecutive jumps is heavy tailed with a divergent 
mean, in a manner similar to Brownian motion for normal diffusion, and the governing equation 
for the probability density function of the particle appearing at certain time instance $t$ 
and space location $x$ is of the form. Subdiffusion has been observed in several applications in
engineering, physics and biology, e.g., thermal diffusion in fractal domains \cite{Nigmatulin:1986}, and 
dispersive ion transport in column experiments \cite{HatanoHatano:1998}; see the review 
\cite{MetzlerKlafter:2000} for physical motivation and an extensive list of physical applications; 
See also the works \cite{HenryLanglandsWearne:2006,YusteAbad:2010} for the derivation of
reaction-subdiffusion models within the framework of continuous time random walk.

In this work, we study the following inverse problem for the model \eqref{eqn:fde}: given a function $g$, recover
$q\in L^2(\Omega)$ such that
\begin{equation}\label{eqn:ip}
   u(q)(\cdot,T)=g\quad \mbox{in }\Omega.
\end{equation}
The direct problem for $q\in L^2(\Omega)$ has not been extensively studied, and we give a study
in Section \ref{sec:cauchy} via an operator theoretic formulation. Let $A=-\Delta$, with its domain
$D(A)=H_0^1(\Omega)\cap H^2(\Omega)$, and the graph norm denoted by $\|\cdot\|_{D(A)}$. If
$\omega$ is an open subset of the domain $\Omega$, we identify
\begin{equation*}
  L^2(\omega) = \{f\in L^2(\Omega): (1-\chi_\omega)f=0\},
\end{equation*}
where $\chi_\omega$ denotes the characteristic function of the subset $\omega$. We denote by $\lambda_1$ the smallest eigenvalue of $A$,
and $\bar\varphi_1$ the corresponding nonnegative eigenfunction, normalized by $\|\bar\varphi_1\|_{L^\infty(\Omega)}=1$. Further, let
\begin{equation}\label{eqn:cal}
  c_\alpha = \sup_{t\geq0}tE_{\alpha,\alpha}(-t),
\end{equation}
where $E_{\alpha,\alpha}(z)$ is the two-parameter Mittag-Leffler function defined in \eqref{eqn:mlf} below.
This constant plays a crucial role in the analysis. Proposition  \ref{prop:bbd-Ealal}
gives an upper bound on $c_\alpha$, which implies $c_\alpha<\alpha$.

Then the following stability estimate holds: for small time $T$, the inverse
problem is locally Lipschitz stable. The proof of the theorem employs the
implicit function theorem, and certain estimates on the solution operators with sharp constants, which in turn uses
heavily refined properties of Mittag-Leffler functions; See Section \ref{sec:inv} for the detailed proof.
\begin{theorem}\label{thm:inverse}
Let $\frac34<\gamma<1$, $0<\epsilon<1-\frac{c_\alpha}{\alpha}$, $\mu_0,\mu_1>0$ such that $1\leq \frac{\mu_1}{\mu_0}< \frac{(1-\epsilon)\alpha}{c_\alpha}$.
Let $u_0\in D(A^{1+\gamma})$, with
\begin{equation}\label{eqn:bdd-u0}
  \mu_0\lambda_1\bar\varphi_1\leq -\Delta u_0\leq \mu_1\lambda_1\bar\varphi_1
\end{equation}
and set $\omega=\{x\in \Omega: \bar\varphi_1(x)\geq 1-\epsilon\}$. Then there exists a constant $\theta>0$ depending only on $\frac{\mu_1}{(1-\epsilon)\mu_0}$
and $\alpha$ such that if $\lambda_1T^\alpha<\theta$, then there is $V$, a neighborhood of $0$ in $L^2(\omega)$ and a constant $c$ such that
\begin{equation*}
  \|q_1-q_2\|_{L^2(\omega)}\leq c\|u(q_1)(T)-u(q_2)(T)\|_{D(A)},\quad \forall q_1,q_2\in V.
\end{equation*}
\end{theorem}

\begin{remark}
The regularity condition $u_0\in D(A^{1+\gamma})$ is to ensure the well-posedness
of the direct problem with $q\in L^2(\Omega)$. The condition \eqref{eqn:bdd-u0} is to ensure pointwise
lower and upper bounds on the solution $u(0)(T)$, and the set of $u_0$ satisfying \eqref{eqn:bdd-u0}
is a convex subset of $D(A^{1+\gamma})$. The condition $\lambda_1T^\alpha<\theta$
dictates that either $T$ or $\lambda_1$ should be sufficiently small, the latter
of which holds if the domain $\Omega$ is large, since $\lambda_1$ tends to zero
as the volume of $\Omega$ tends to infinity \cite{CourantHilbert:1953}.
\end{remark}

We also develop a simple algorithm to numerically recover the potential $q$. It is based on preconditioned
fixed point iteration given in \eqref{eqn:fp}, and employs Anderson acceleration \cite{Anderson:1965} to speed up the
convergence. It extends an existing scheme proposed in \cite{Rundell:1987} for the standard parabolic problem to
subdiffusion, but
enhanced by the preconditioner $A^{-1}$ for better numerical stability and acceleration via Anderson acceleration.
The algorithm is straightforward to implement, since it involves solving one direct problem at each iteration,
and generally applicable (no sign restriction, no condition on the initial data), and when equipped with the
discrepancy principle \cite{EnglHankeNeubauer:1996,ItoJin:2015}, it is also accurate for both subdiffusion and normal
diffusion. We provide several numerical experiments to confirm the efficiency and accuracy of the algorithm, and to
illustrate the behavior of the inverse problem. The
stability result in Theorem \ref{thm:inverse} and the reconstruction algorithm represent the main contributions of this work.

Now we discuss several existing works. Inverse problems for subdiffusion are of relative recent nature,
initiated by the pioneering work \cite{ChengNakagawaYamamotoYamazaki:2009} for recovering the diffusion coefficient
from lateral Cauchy data (in the one-dimensional case) using Sturm-Liouville theory; see the work
\cite{JinRundell:2015} for an overview. The inverse potential problem for the model \eqref{eqn:fde} has
also been analyzed in several works \cite{JinRundell:2012,MillerYamamoto:2013,ZhangZhou:2017,
KaltenbacherRundell:2019,KianYamamoto:2019}. Miller and Yamamoto \cite{MillerYamamoto:2013} proved the unique
recovery from data on a space-time subdomain, using an integral transformation. Zhang and Zhou \cite{ZhangZhou:2017}
discussed the unique recovery using a fixed point argument \cite{Isakov:1991}, and derived error estimates in
the presence of data noise. Kaltenbacher and Rundell \cite{KaltenbacherRundell:2019} gave the well-posedness
of the direct problem and also proved the invertibility
of the linearized map from the space $L^2(\Omega)$ to $H^2(\Omega)$ under the condition $u_0>0$ in $\Omega$
and $q\in L^\infty(\Omega)$ using a Paley-Wiener type result, where the condition $q\in L^\infty(\Omega)$ plays
a central in the proof, which invokes a type of strong maximum principle. Further, they developed (frozen) Newton and Halley
type iterative schemes for numerically recovering the coefficient $q$ from the terminal data, and proved their convergence. Kian and
Yamamoto \cite{KianYamamoto:2019} derived a stability result for recovering a space-time dependent potential coefficient from lateral
Cauchy data. It is worth noting that the parabolic counterpart of the inverse problem \eqref{eqn:ip} has been extensively
studied \cite{PrilepkoSolovev:1987,Isakov:1991,ChoulliYamamoto:1996,ChoulliYamamoto:1997,KlibanovLiZhang:2020}.
Isakov \cite{Isakov:1991} proved the
existence and uniqueness for the inverse problem using strong maximum principle, and proposed a constructive
algorithm based on fixed point iteration. Choulli and Yamamoto \cite{ChoulliYamamoto:1996} proved a generic
well-posedness result in a Holderian space, by introducing a scalar parameter in the leading elliptic term $\Delta u$.
Later, they \cite{ChoulliYamamoto:1997} analyzed the inverse problem in a Hilbert space setting.
Theorem \ref{thm:inverse} represents an extension of the result in \cite{ChoulliYamamoto:1997} to the subdiffusion case. Note
that due to the drastic difference in solution operators, i.e., the fractional case involves Mittag-Leffler
functions, the extension is nontrivial. We refer interested readers
to \cite{SakamotoYamamoto:2011isp,LiuLiYamamoto:2019isp,LiChengLiu:2020} and references therein
for related inverse source problems, which are often employed to analyze the generic well-posedness
for the inverse potential problem.

The rest of the paper is organized as follows. In Section \ref{sec:cauchy}, we discuss the
well-posedness of the direct problem, and prove that for every $q\in L^2(\Omega)$, there exists a
unique classical solution, for suitably smooth initial data $u_0$. Then in Section \ref{sec:inv}, we
give the proof of Theorem \ref{thm:inverse}. Next, we develop the fixed point algorithm and
present its preliminary properties in
Section \ref{sec:alg}. Last, we provide several numerical experiments to illustrate
feasibility of the reconstruction algorithm. Throughout, $(\cdot,\cdot)$
denotes the $L^2(\Omega)$ inner product, and $H^s(\Omega)$ denotes the usual Sobolev space \cite{AdamsFournier:2003}. The notation
$c$ denotes a generic constant which may change at each occurrence, but it is always independent of
the coefficient $q$.

\section{Well-posedness of the Cauchy problem}\label{sec:cauchy}
First we study the well-posedness of the following abstract Cauchy problem:
\begin{equation}\label{eqn:fde-op}
  \left\{\begin{aligned}
  \partial_t^\alpha u (t) +Au(t) &= qu(t),\quad \mbox{in }(0,T),\\
  u(0) & =u_0
  \end{aligned}\right.
\end{equation}
It is a reformulation of the direct problem \eqref{eqn:fde} into an operator form. We prove that for suitably smooth
$u_0$ and any $q\in L^2(\Omega)$, problem \eqref{eqn:fde-op} has a unique classical solution $u=u(q)\in
C^\alpha([0,T];L^2(\Omega))\cap C([0,T];D(A))$. The analysis is based on a ``perturbation'' argument, developed
recently in \cite{JinLiZhou:2018} for the numerical analysis of nonlinear subdiffusion problems, where Banach
fixed point theorem and the argument
of equivalent norm family play an important role (see, e.g., \cite[Chapter 3]{Ciarlet:2013}); See also \cite{KaltenbacherRundell:2019}
for a well-posedness result under slightly different assumptions on the potential $q$.

Specifically, let $\{(\lambda_j,\varphi_j)\}_{j=1}^\infty$ be the eigenpairs of the operator $A$, with
the eigenvalues $\lambda_j$ ordered nondecreasingly and multiplicity counted, and $\{\varphi_j\}_{j=1}^\infty$ form an
orthonormal basis of $L^2(\Omega)$. For any $\gamma>0$, the notation $D(A^\gamma)$ denotes the domain of
the fractional power $A^\gamma$, with the graph norm $\|\cdot\|_{D(A^\gamma)}$, given by
\begin{equation*}
  \|v\|_{D(A^\gamma)} = \Big(\sum_{j=1}^\infty \lambda_j^{2\gamma} (v,\varphi_j)^2\Big)^\frac12.
\end{equation*}

 By viewing $qu(t)$ as the inhomogeneous term and applying Duhamel's
principle, we deduce that the solution $u(t)$ satisfies
\begin{equation*}
  u(t) = U(t) + \int_0^tE(t-s)qu(s)\d s,
\end{equation*}
where $U(t)=F(t)u_0$, and the solution operators $F(t)$ and $E(t)$ are defined by \cite{JinLiZhou:2018}
\begin{align*}
  F(t)v &= \sum_{j=1}^\infty E_{\alpha,1}(-\lambda_jt^\alpha)(\varphi_j,v)\varphi_j,\\
  E (t)v &= \sum_{j=1}^\infty t^{\alpha-1}E_{\alpha,\alpha}(-\lambda_jt^\alpha)(\varphi_j,v)\varphi_j.
\end{align*}
Here $E_{\alpha,\beta}(z)$ refers to the two-parameter Mittag-Leffler function,
defined by \cite{KilbasSrivastavaTrujillo:2006}
\begin{equation}\label{eqn:mlf}
  E_{\alpha,\beta}(z) = \sum_{k=0}^\infty \frac{z^k}{\Gamma(k\alpha+\beta)},\quad z\in \mathbb{C}.
\end{equation}

The next lemma collects smoothing properties of the operators $F$ and $E$.
The notation $\|\cdot\|$ denotes the operator norm on $L^2(\Omega)$.
\begin{lemma}\label{lem:solop}
For the operators $F(t)$ and $E(t)$, the following estimates hold
\begin{align*}
  \|F(t)\|\leq E_{\alpha,1}(-\lambda_1t^\alpha),\quad
  \|E(t)\|\leq t^{\alpha-1}E_{\alpha,\alpha}(-\lambda_1t^\alpha),\quad
  \|A^\theta E(t)\|\leq c_{\alpha,\theta}t^{(1-\theta)\alpha-1}.
\end{align*}
where the constant $c_{\alpha,\theta}>0$ depends on $\alpha$ and $\theta$.
\end{lemma}
\begin{proof}
The first estimate follows from the fact that $E_{\alpha,1}(-t)$ is completely monotone \cite{Pollard:1948}:
\begin{align*}
  &\|F(t)v\|_{L^2(\Omega)}^2 = \sum_{j=1}^\infty E_{\alpha,1}(-\lambda_jt^\alpha)^2(\varphi_j,v)^2 \\
  \leq & E_{\alpha,1}(-\lambda_1t^\alpha)^2\sum_{j=1}^\infty (\varphi_j,v)^2 = E_{\alpha,1}(-\lambda_j t^\alpha)^2\|v\|_{L^2(\Omega)}^2.
\end{align*}
The second follows similarly since $E_{\alpha,\alpha}(-t)$ is also
completely monotone, and the last is known from \cite[Lemma 3.4]{JinLiZhou:2018}.
\end{proof}

Now we can specify the function analytic setting. Let $0<\beta<(1-\gamma)\alpha$ be fixed and set
\begin{equation*}
  X = C^\beta([0,T];D(A^\gamma)) \cap C([0,T];D(A)),
\end{equation*}
with the norm given by
\begin{equation*}
  \|v\|_X=\|v\|_{C^\beta([0,T];D(A^\gamma))}+\|v\|_{C([0,T];D(A))}.
\end{equation*}

Then for every $q\in L^2(\Omega)$, we define an associated operator $L(q)$ by
\begin{equation*}
  [L(q)]f (t) = \int_0^t E(t-s)qf(s)\d s,\quad \forall f\in C^\beta([0,T];D(A^\gamma)).
\end{equation*}
The next result gives the mapping property of the operator $L(q)$.
\begin{lemma}
For any $q\in L^2(\Omega)$, $L(q)$ maps $C^\beta([0,T];D(A^\gamma))$, with $\gamma>\frac34$, into $X$.
\end{lemma}
\begin{proof}
Let $f\in C^\beta([0,T];D(A^\gamma))$, $q\in L^2(\Omega)$ and let $g(t)=qf(t)$, $0\leq t\leq T.$
We split the function $L(q)f$ into two terms $L(q)f=v_1+v_2$, with
\begin{equation*}
  v_1(t) = \int_0^tE(t-s)(g(s)-g(t))\d s\quad \mbox{and}\quad v_2(t) = \int_0^tE(t-s)g(t)\d s.
\end{equation*}
Since $f\in C^\beta([0,T];D(A^\gamma))$, by Sobolev embedding theorem \cite{AdamsFournier:2003}, $g\in C^\beta([0,T];L^2(\Omega))$, and thus by
\cite[Lemma 3.4]{SakamotoYamamoto:2011}, $v_1\in C^\beta([0,T];D(A))\subset X$.
Next, for $t\in [0,T]$, $\tau>0$ such that $t+\tau\leq T$, we have
\begin{equation*}
  v_2(t+\tau)-v_2(t) = \int_{t}^{t+\tau}E(s)g(t+\tau)\d s + \int_0^tE(s)[g(t+\tau)-g(t)]\d s.
\end{equation*}
Thus by the smoothing property of $E$ in Lemma \ref{lem:solop}, we deduce
\begin{align*}
   &\quad \|A^\gamma(v_2(t+\tau)-v_2(t))\|_{L^2(\Omega)}&\\
  \leq& c_{\alpha,\gamma}\|g\|_{C([0,T];L^2(\Omega))}
  \int_t^{t+\tau}s^{(1-\gamma)\alpha-1}\d s + c_{\alpha,\gamma}\tau^\beta\|g\|_{C^\beta([0,T];L^2(\Omega))}\int_0^ts^{(1-\gamma)\alpha-1}\d s.
\end{align*}
Since
\begin{equation}\label{eqn:bdd-int}
  \int_t^{t+\tau}s^{(1-\gamma)\alpha-1}\d s \leq \int_0^\tau s^{(1-\gamma)\alpha-1}\d s = \frac{\tau^{(1-\gamma)\alpha}}{(1-\gamma)\alpha},
\end{equation}
we obtain
\begin{equation*}
  \tau^{-\beta}\|A^\gamma[v_2(t+\tau)-v_2(t)]\|_{L^2(\Omega)} \leq \frac{c_{\alpha,\gamma}}{(1-\gamma)\alpha} (
  \tau^{(1-\gamma)\alpha-\beta}+ T^{(1-\gamma)\alpha})\|g\|_{C^\beta([0,T];L^2(\Omega))}.
\end{equation*}
Since $\beta<(1-\gamma)\alpha$, $v_2\in C^\beta([0,T];L^2(\Omega))$. It remains to show $Av_2\in C([0,T];L^2(\Omega))$.
This follows from the identity
\begin{equation*}
  -Av_2(t) = -\int_0^tAE(t-s)g(t)\d s = (F(t)-I)g(t),
\end{equation*}
in view of the identity $\frac{\d}{\d t}(I-F(t))=AE(t)$ \cite{JinLiZhou:2018}. Since $F(t)-I$ is continuous on $L^2(\Omega)$,
the desired assertion follows. This completes the proof of the lemma.
\end{proof}

\begin{lemma}\label{lem:Lq}
If $q\in L^2(\Omega)$, then $I-L(q)$ has a bounded inverse in $B(X)$.
\end{lemma}
\begin{proof}
The proof proceeds by the argument of equivalent norm family (see, e.g., \cite[Chapter 3, Section 3.8]{Ciarlet:2013}). Specifically,
we equip the space $X$ with an equivalent family of norms $\|\cdot\|_{\lambda}$, $\lambda\geq0$, defined by
\begin{align*}
  \|f\|_\lambda &= \sup_{t\in [0,T]}e^{-\lambda t}[\|f(t)\|_{L^2(\Omega)}+\|A^\gamma f(t)\|_{L^2(\Omega)}] \\
   &\quad + \sup_{0\leq s<t\leq T}e^{-\lambda(t+1)}\frac{\|f(s)-f(t)\|_{D(A^\gamma)}}{|t-s|^\beta}+\sup_{t\in[0,T]}e^{-\lambda(t+2)}\|Af(t)\|_{L^2(\Omega)},
\end{align*}
which is equivalent to the norm on $X$, and then prove the invertibility by choosing $\lambda$
suitably. For $f\in X$, let $v=L(q)f$. Then by Sobolev embedding \cite{AdamsFournier:2003} and Lemma \ref{lem:solop},
\begin{align*}
  e^{-\lambda t}\|v(t)\|_{L^2(\Omega)} &= e^{-\lambda t} \|\int_0^t E(t-s)qf(s)\d s \|_{L^2(\Omega)}\\
    & \leq  c\int_0^t e^{-\lambda (t-s)}\|E(t-s)\| \|q\|_{L^2(\Omega)}e^{-\lambda s}\|f(s)\|_{C([0,T];D(A^\gamma))} \d s\\
    & \leq c\int_0^te^{-\lambda s}s^{\alpha-1}\d s \|q\|_{L^2(\Omega)}\|f\|_{\lambda}\leq c\lambda^{-\alpha}\|q\|_{L^2(\Omega)}\|f\|_\lambda,
\end{align*}
where the last inequality follows from changing variables $\zeta=\lambda s$ by
\begin{align*}
  \int_0^t e^{-\lambda s}s^{\alpha-1}\d s &= \lambda^{-\alpha}\int_0^{\lambda t}e^{-\zeta}\zeta^{\alpha-1}\d \zeta
     \leq \lambda^{-\alpha}\int_0^\infty \zeta^{\alpha-1}e^{-\zeta}\d \zeta = \lambda^{-\alpha}\Gamma(\alpha).
\end{align*}
Similarly,
\begin{align*}
  e^{-\lambda t}\|A^\gamma v(t)\|_{L^2(\Omega)} & \leq c\int_0^te^{-\lambda (t-s)}\|A^\gamma E(t-s)\|\|q\|_{L^2(\Omega)}e^{-\lambda s}\|f(s)\|_{C([0,T];D(A^\gamma))}\d s\\
       & \leq c\int_0^t e^{-\lambda s}s^{(1-\gamma)\alpha-1}\d s\|q\|_{L^2(\Omega)}\|f\|_\lambda\leq c\lambda^{-(1-\gamma)\alpha}\|q\|_{L^2(\Omega)}\|f\|_\lambda.
\end{align*}
Meanwhile, for $t\in [0,T)$ and $\tau>0$ with $t+\tau\leq T$, we have
\begin{equation*}
  A^\gamma(v(t+\tau)-v(t)) = \int_0^t A^\gamma E(s)q[f(t+\tau-s)-f(t-s)]\d s + \int_t^{t+\tau}A^\gamma E(s)qf(t+\tau-s)\d s,
\end{equation*}
which directly implies
\begin{align*}
 &\quad e^{-\lambda(t+\tau+1)}\|A^\gamma (v(t+\tau)-v(t))\|_{L^2(\Omega)}\\
  &\leq c\int_0^te^{-\lambda s}\|A^\gamma E(s)\|\|q\|_{L^2(\Omega)}e^{-\lambda(t+\tau-s+1)}\|f(t+\tau-s)-f(t-s)\|_{D(A^\gamma)}\d s\\
      &\quad +c\int_t^{t+\tau}e^{-\lambda(s+1)}\|A^\gamma E(s)\|\|q\|_{L^2(\Omega)}e^{-\lambda(t+\tau-s)}\|f(t+\tau-s)\|_{D(A^\gamma)}\d s\\
      &\leq c\tau^\beta\|q\|_{L^2(\Omega)}\|f\|_{\lambda}\Big(\int_0^t e^{-\lambda s}s^{(1-\gamma)\alpha-1}\d s +
       \tau^{-\beta} \int_t^{t+\tau}e^{-\lambda(s+1)}s^{(1-\gamma)\alpha-1}\d s\Big).
\end{align*}
This, the inequality \eqref{eqn:bdd-int} and the choice $\beta<(1-\gamma)\alpha$ give
\begin{align*}
e^{-\lambda(t+\tau+1)}\tau^{-\beta}\|A^\gamma (v(t+\tau)-v(t))\|_{L^2(\Omega)} &\leq c\|q\|_{L^2(\Omega)}\|f\|_{\lambda}\big(\lambda^{-(1-\gamma)\alpha} +
       \tau^{(1-\lambda)\alpha-\beta} e^{-\lambda}\big)\\
       &\leq c\|q\|_{L^2(\Omega)}\|f\|_{\lambda}\big(\lambda^{-(1-\gamma)\alpha} + e^{-\lambda}\big).
\end{align*}
In the same way, we deduce
\begin{equation*}
  \tau^{-\beta}e^{-\lambda(t+\tau+1)}\|v(t+\tau)-v(t)\|_{L^2(\Omega)}\leq c_T\|q\|_{L^2(\Omega)}\|f\|_\lambda(\lambda^{-\alpha}+e^{-\lambda}).
\end{equation*}
Combining the preceding two estimates gives
\begin{equation*}
  \sup_{0\leq s<t\leq T}e^{-\lambda(t+1)}\frac{\|v(t)-v(s)\|_{D(A^\gamma)}}{|t-s|^{\beta}} \leq c_T\|q\|_{L^2(\Omega)}\|f\|_\lambda (\lambda^{-(1-\gamma)\alpha}+e^{-\lambda}).
\end{equation*}
Next, in view of the identities $\frac{\d}{\d t}(I-F(t))=AE(t)$ \cite{JinLiZhou:2018} and $\lim_{t\to0^+}\|F(t)-I\|=0$, we deduce
\begin{equation*}
  -Av(t) = \int_0^t-AE(t-s)q[f(s)-f(t)]+F(t)qf(t)-qf(t).
\end{equation*}
Then the preceding argument and Lemma \ref{lem:solop} lead to
\begin{align*}
  e^{-\lambda(t+2)}\|Av(t)\|_{L^2(\Omega)}&\leq c \int_0^te^{-\lambda}\|AE(t-s)\|\|q\|_{L^2(\Omega)}e^{-\lambda(t+1)}\|f(s)-f(t)\|_{D(A^\gamma)}\d s\\
     & \quad + ce^{-\lambda(t+2)}\|q\|_{L^2(\Omega)}\|f(t)\|_{D(A^\gamma)}\\
     &\leq c \|q\|_{L^2(\Omega)}\|f\|_\lambda e^{-\lambda}\Big(\int_0^t(t-s)^{\beta-1}\d s+ 1\Big) \leq c_T\|q\|_{L^2(\Omega)} \|f\|_\lambda e^{-\lambda}.
\end{align*}
Combining the preceding estimates implies
\begin{equation*}
   \|L(q)f\|_\lambda\leq c_T(e^{-\lambda} +\lambda^{-(1-\gamma)\alpha})\|q\|_{L^2(\Omega)}\|f\|_\lambda.
\end{equation*}
It follows directly from this estimate that the function $\|L(q)f\|_\lambda$ tends to zero as $\lambda$
tends to infinity, and thus the operator norm $\|L(q)\|_\lambda<1$ if $\lambda$ is large enough, which shows the lemma.
\end{proof}

Now we can state the unique solvability of the Cauchy problem \eqref{eqn:fde-op}.
\begin{proposition}
If $q\in L^2(\Omega)$ and $u_0\in D(A^{1+\gamma})$. Then the Cauchy problem \eqref{eqn:fde-op} has a unique classical solution $u(q)=(I-L(q))^{-1}U$.
\end{proposition}
\begin{proof}
Since $u_0\in D(A^{1+\gamma})$, by the regularity theory for subdiffusion \cite[Theorems 2.1 and 2.3]{SakamotoYamamoto:2011},
$U\in X$. Thus, by Lemma \ref{lem:Lq}, the Cauchy problem \eqref{eqn:fde-op} has a unique solution $u=(I-L(q))^{-1}U
\in X$. The fact that $u$ is the classical solution to problem \eqref{eqn:fde-op} follows from the fact that $qu
\in C^\beta([0,T];L^2(\Omega))$ and $u_0\in D(A)$, by the regularity theory of subdiffusion \cite{SakamotoYamamoto:2011}.
\end{proof}

\section{Proof of Theorem \ref{thm:inverse}}\label{sec:inv}

Below we assume $u_0$ satisfies the condition of Theorem \ref{thm:inverse}. In view of the Sobolev embedding
$D(A^{1+\gamma})\hookrightarrow C^{2,\delta}(\overline{\Omega})$ for some $\delta>0$, the function $U(t)\in C^{2+\delta,
\frac{2+\delta}{2}\alpha}(\overline{\Omega}\times[0,T])$ \cite{Krasnoschok:2016}, and satisfies
\begin{equation*}
  \left\{\begin{aligned}
  \partial_t^\alpha U&= \Delta U, \quad\mbox{in } \Omega\times(0,T],\\
  U(0) &= u_0, \quad \mbox{in }\Omega,\\
  U &= 0, \quad\mbox{on }\partial\Omega\times[0,T].
  \end{aligned}\right.
\end{equation*}
The next result collect several properties of the function $U(t)$.
\begin{lemma}\label{lem:bdd-U}
The following properties hold on the function $U(t)$.
\begin{itemize}
  \item[$\rm(i)$] $\mu_0\bar\varphi_1(x)\leq u_0(x)\leq \mu_1\bar\varphi_1(x)$, $x\in \overline{\Omega}$.
  \item[$\rm(ii)$] $\mu_0E_{\alpha,1}(-\lambda_1t^\alpha)\bar\varphi_1(x)\leq U(x,t)\leq \mu_1E_{\alpha,1}(-\lambda_1t^\alpha)\bar\varphi_1(x)$, $(x,t)\in \overline{\Omega}\times[0,T]$.
  \item[$\rm(iii)$] $0\leq -\partial_t^\alpha U(x,t)\leq \|\Delta u_0\|_{L^\infty(\Omega)}\leq \mu_1\lambda_1$, $(x,t)\in\Omega\times[0,T]$
\end{itemize}
\end{lemma}
\begin{proof}
Part (i) is already proved in \cite{ChoulliYamamoto:1997}. Part (ii) follows from the maximum principle for
the subdiffusion model (see, e.g., \cite[Theorem 1.1]{YikanRundellYamamoto:2016} or \cite{LuchkoYamamoto:2017}).
We only prove (iii). Let $w(x,t)=\partial_t^\alpha U(x,t)$. Then $w$ satisfies
\begin{equation*}
   \left\{\begin{aligned}
  \partial_t^\alpha w&= \Delta w, \quad\mbox{in } \Omega\times(0,T],\\
  w(0) &= \Delta u_0, \quad \mbox{in }\Omega,\\
  w &= 0, \quad\mbox{on }\partial\Omega\times[0,T].
  \end{aligned}\right.
\end{equation*}
By assumption, $\Delta u_0\leq 0$ in $\Omega$, and thus by the maximum principle for the subdiffusion
\cite[Theorem 1.1]{YikanRundellYamamoto:2016}, $0\leq -w(x,t)\leq \|\Delta u_0\|_{L^\infty(\Omega)}$.
This implies assertion (iii).
\end{proof}

Let $\omega$ be defined as in Theorem \ref{thm:inverse}. Lemma \ref{lem:bdd-U}(ii) implies that
\begin{equation*}
   u_T=\frac{1}{U(T)|_\omega}
\end{equation*}
extended by zero outside $\omega$ belongs to $L^\infty(\Omega)$.
Now we define the operator $P_T:L^2(\omega)\to L^2(\omega)$ by
\begin{equation*}
  q \mapsto \int_0^T-AE(T-s)u_T[U(s)-U(T)]q\,\d s + F(T)q.
\end{equation*}
This operator arises in the linearization of the forward map.

The next result gives an upper bound on the constant $c_\alpha$ defined in \eqref{eqn:cal}.
In particular, it indicates that $c_\alpha< \alpha<1$, which is crucial for proving Theorem \ref{thm:inverse}.
\begin{proposition}\label{prop:bbd-Ealal}
For any $\alpha\in (0,1)$,
\begin{equation*}
  c_\alpha :=\sup_{t\geq 0}tE_{\alpha,\alpha}(-t)  \leq
\frac{\alpha^2\pi}{\sin(\alpha\pi)+\alpha\pi}.
\end{equation*}
\end{proposition}
\begin{proof}
Let $f(t)=t E_{\alpha,\alpha}(-t)$. By the asymptotics of the Mittag-Leffler function $E_{\alpha,\alpha}(-t)$, i.e., $E_{\alpha,\alpha}(-t)\leq
\frac{c}{1+t^2}$, and complete monotonicity of the function $E_{\alpha,\alpha}(-t)$,
the function $f(t)$ is nonnegative on $[0,\infty)$, and tends to zero as $t\to\infty$. Thus, there exists a maximum.
Now let $u(t)=t^{\alpha-1}E_{\alpha,\alpha}(- t^\alpha)$. Then it satisfies the following ODE
\begin{equation*}
  ^R\partial_t^\alpha u + u =0,\quad t>0,\qquad \mbox{with }\quad {_0I_t^{1-\alpha}}u(0)=1,
\end{equation*}
where the notation $^R\partial_t^\alpha$ and ${_0I_t^\beta}$ denote the Riemann-Liouville fractional derivative
and integral, respectively, based at $t=0$. Let $w(t)=tu(t)=t^\alpha E_{\alpha,\alpha}(-t^\alpha)$. Then direct
computation with the identity $^R\partial_t^\alpha t^\gamma = \frac{\Gamma(\gamma+1)}{\Gamma(\gamma-\alpha+1)}t^{\gamma-\alpha}$ leads to
\begin{align*}
   ^R\partial_t^\alpha w(t) & =  \sum_{k=0}^\infty \frac{(-1)^k}{\Gamma(k\alpha+\alpha)}{^R\partial_t^\alpha}t^{k\alpha+\alpha} =  \sum_{k=0}^\infty \frac{(-1)^k}{\Gamma(k\alpha+\alpha)}\frac{\Gamma(k\alpha+\alpha+1)}{\Gamma(k\alpha+1)}t^{k\alpha}.
\end{align*}
Using the recursion $\Gamma(z+1)=z\Gamma(z)$ twice,
\begin{align*}
               ^R\partial_t^\alpha w(t)  & =  \sum_{k=0}^\infty \frac{k\alpha(-1)^k}{\Gamma(k\alpha+1)}t^{k\alpha} + \alpha \sum_{k=0}^\infty \frac{(-1)^k}{\Gamma(k\alpha+1)}t^{k\alpha}\\
                            & =  \sum_{k=0}^\infty \frac{(-1)^k}{\Gamma(k\alpha)}t^{k\alpha} + \alpha \sum_{k=0}^\infty \frac{(-1)^k}{\Gamma(k\alpha+1)}t^{k\alpha}\\
                            & = - t^\alpha E_{\alpha,\alpha}(-t^\alpha) + \alpha E_{\alpha,1}(-t^\alpha),
\end{align*}
where the last step follows since $1/\Gamma(0)=0$. Consequently,
\begin{equation*}
  ^R\partial_t^\alpha w + w = \alpha E_{\alpha,1}(-t^\alpha)\quad \mbox{with } {_0I_t^{1-\alpha}}w(0)=0.
\end{equation*}
Thus, the solution theory for fractional ODEs indicates that $w(t)$ is represented by
\begin{equation}\label{eqn:w-repres}
  w(t) = \int_0^t (t-s)^{\alpha-1}E_{\alpha,\alpha}(-(t-s)^{\alpha})\alpha E_{\alpha,1}(-s^{\alpha})\d s.
\end{equation}
Then using the facts that $0\leq E_{\alpha,1}(-t)\leq 1$, $E_{\alpha,\alpha}(-t)\geq0$ (as a result of the
complete monotonicity of $E_{\alpha,1}(-t)$ \cite{Pollard:1948}), and the differentiation formula
\begin{equation*}
  \frac{\d}{\d t}E_{\alpha,1}(-t^\alpha) = -t^{\alpha-1}E_{\alpha,\alpha}(-t^\alpha),
\end{equation*}
we deduce
\begin{align*}
  w(t) & \leq \alpha \int_0^t (t-s)^{\alpha-1}E_{\alpha,\alpha}(-(t-s)^{\alpha}) \d s = \alpha (1-E_{\alpha,1}(-t^\alpha)) < \alpha.
\end{align*}
Meanwhile, by Simon's theorem \cite{Simon:2014},
\begin{equation*}
  E_{\alpha,1}(-t^\alpha) \leq \frac{1}{1+\Gamma(1+\alpha)^{-1}t^\alpha}<\Gamma(1+\alpha)t^{-\alpha},
\end{equation*}
 there holds
\begin{align*}
  w(t) &= \int_0^t (t-s)^{\alpha-1}E_{\alpha,\alpha}(-(t-s)^{\alpha})\alpha E_{\alpha,1}(-s^{\alpha})\d s\\
       & \leq \alpha\Gamma(\alpha+1) \int_0^t (t-s)^{\alpha-1}E_{\alpha,\alpha}(-(t-s)^\alpha)s^{-\alpha}\d s\\
       & = \alpha\Gamma(\alpha+1) \int_0^t\sum_{k=0}^\infty \frac{(-1)^k}{\Gamma(k\alpha+\alpha)}(t-s)^{k\alpha+\alpha-1}s^{-\alpha}\d s.
\end{align*}
Using the identity $\int_0^t (t-s)^{a-1}s^{-b}\d s = t^{a-b}\frac{\Gamma(a)\Gamma(1-b)}{\Gamma(a+1-b)}$ for any $a>0$ and $b<1$,
we deduce
\begin{align*}
  w(t) &\leq \alpha\Gamma(\alpha+1) \sum_{k=0}^\infty \frac{(-1)^k}{\Gamma(k\alpha+\alpha)}\frac{\Gamma(k\alpha+\alpha)\Gamma(1-\alpha)}{\Gamma(k\alpha+1)}t^{k\alpha}\\
       &\leq \alpha\Gamma(\alpha+1)\Gamma(1-\alpha) \sum_{k=0}^\infty \frac{(-1)^k}{\Gamma(k\alpha+1)}t^{k\alpha} = \alpha\Gamma(1-\alpha)\Gamma(1+\alpha)E_{\alpha,1}(-t^\alpha).
\end{align*}
Now by the recursion identity and reflection identity for the Gamma function,
\begin{equation*}
  \Gamma(1-\alpha)\Gamma(1+\alpha)=\alpha\Gamma(1-\alpha)\Gamma(\alpha) = \frac{\alpha\pi}{\sin(\alpha\pi)}.
\end{equation*}
Combining the preceding estimates leads directly to
\begin{equation}\label{eqn:Ealal-sup}
  \sup_{t\geq 0} tE_{\alpha,\alpha}(-t) \leq \alpha \max_{t\ge0}\min\Big(\frac{\alpha\pi}{\sin(\alpha\pi)}E_{\alpha,1}(-t^\alpha),1-E_{\alpha,1}(-t^\alpha)\Big).
\end{equation}
By the complete monotonicity of $E_{\alpha,1}(-t)$, the first term $\frac{\alpha\pi}{\sin(\alpha\pi)}E_{\alpha,1}(-t^\alpha)$
in the bracket is monotonically decreasing, whereas the second term $1-E_{\alpha,1}(-t^\alpha)$ is monotonically increasing.
Thus, one simple upper bound is obtained by equating these two terms, which directly gives
\begin{equation*}
E_{\alpha,1}(-t_*^\alpha)  = \frac{1}{1+\frac{\alpha\pi}{\sin(\alpha\pi)}}.
\end{equation*}
Upon substituting it back to \eqref{eqn:Ealal-sup} and noting the complete monotonicity of $E_{\alpha,1}(-t)$, we deduce
\begin{equation*}
  c_\alpha \leq \alpha \frac{\frac{\alpha\pi}{\sin(\alpha\pi)}}{1+\frac{\alpha\pi}{\sin(\alpha\pi)}} = \frac{\alpha^2\pi}{\alpha\pi + \sin (\alpha\pi)}.
\end{equation*}
This completes the proof of the proposition.
\end{proof}

\begin{remark}
Note that the identities
\begin{align*}
 \lim_{\alpha\to0^+} \frac{\alpha\pi}{\alpha\pi+\sin\alpha\pi } = \frac12\quad\mbox{and}\quad
 \lim_{\alpha\to1^-} \frac{\alpha\pi}{\alpha\pi+\sin\alpha\pi } = 1.
\end{align*}
and the function $f(\alpha)=\frac{\alpha\pi}{\alpha\pi+\sin\alpha\pi}$ is strictly increasing
in $\alpha$ over the interval $(0,1)$. Thus, the factor is strictly less than 1 for any $\alpha\in (0,1)$.
Note also that for the limiting case $\alpha=1$, the constant $c_1=\sup_{t\geq 0}te^{-t}=e^{-1}$,
which is much sharper than the preceding bound. Since the function $E_{\alpha,\alpha}(-t)$ is actually continuous
in $\alpha$, one may refine the bound on $c_\alpha$ slightly for $\alpha$ close to unit. Further,
it is worth noting that the integral representation \eqref{eqn:w-repres} for $w(t)$ can also be deduced from
the following Cristoffel-Darboux type formula for Mittag-Leffler functions, i.e.,
\begin{align*}
  \int_0^ts^{\gamma-1}E_{\alpha,\gamma}(ys^\alpha)(t-s)^{\beta-1}E_{\alpha,\beta}(z(t-s)^\alpha)\d s
   = \frac{yE_{\alpha,\gamma+\beta}(yt^\alpha)-zE_{\alpha,\gamma+\beta}(zt^\alpha)}{y-z}t^{\gamma+\beta-1},
\end{align*}
where $y\neq z$ are any complex numbers. Consequently, by a limiting argument,
\begin{align*}
\int_0^ts^{\alpha-1}E_{\alpha,\alpha}(-s^\alpha)E_{\alpha,1}(-(t-s)^\alpha)\d s
   = \frac{\d}{\d\lambda} \lambda t^\alpha E_{\alpha,\alpha+1}(\lambda t^\alpha)|_{\lambda=-1},
\end{align*}
which upon simplification gives directly the formula \eqref{eqn:w-repres} for $w(t)$.
\end{remark}

\begin{remark}
Proposition \ref{prop:bbd-Ealal} provides an upper bound on the constant $c_\alpha$. In Fig. \ref{fig:cal}(a),
we plot the function $\alpha^{-1}tE_{\alpha,\alpha}(-t)$ versus $t$ for several different
fractional orders, where the Mittag-Leffler function $E_{\alpha,\alpha}(-t)$ is computed using an algorithm
developed in \cite{SeyboldHilfer:2008}.
Clearly, for any fixed $\alpha$, the function $tE_{\alpha,\alpha}(-t)$ first increases
with $t$ and then decreases, and there is only one global maximum. The maximum is
always achieved at some $t^*$ between $0.8$ and $1$, a fact that remains to be established, and the maximum value decreases
with $\alpha$. The optimal constant $\frac{c_\alpha}{\alpha}$ versus the upper bound
$\frac{\alpha\pi}{\sin\alpha\pi + \alpha\pi}$ is shown in Fig. \ref{fig:cal}(b). Note that $\frac{c_\alpha}{\alpha}$ is strictly increasing
with respect to $\alpha$, and the upper bound in Proposition \ref{prop:bbd-Ealal} is about
three times larger than the optimal one $\frac{c_\alpha}{\alpha}$. This is attributed to
the fact that the derivation employs upper bounds of the Mittag-Leffler function $E_{\alpha,1}(-t)$ that are
valid on the entire real line, instead of sharper ones on a finite interval, e.g., $[0,1]$.
The fact that the ratio $\frac{c_\alpha}{\alpha}$ increases with $\alpha$ implies that the
smaller the fractional order $\alpha$ is, there is a larger degree of freedom for choosing
the parameter $\epsilon$ as well as $\mu_1/\mu_0$ in Theorem \ref{thm:inverse}, which partly indicates
the potential beneficial effect of subdiffusion on the inverse potential problem.
\begin{figure}[hbt!]
\centering
\setlength{\tabcolsep}{0pt}
  \begin{tabular}{cc}
    \includegraphics[width=0.45\textwidth,trim={1.5cm 0 1cm 0.5cm}]{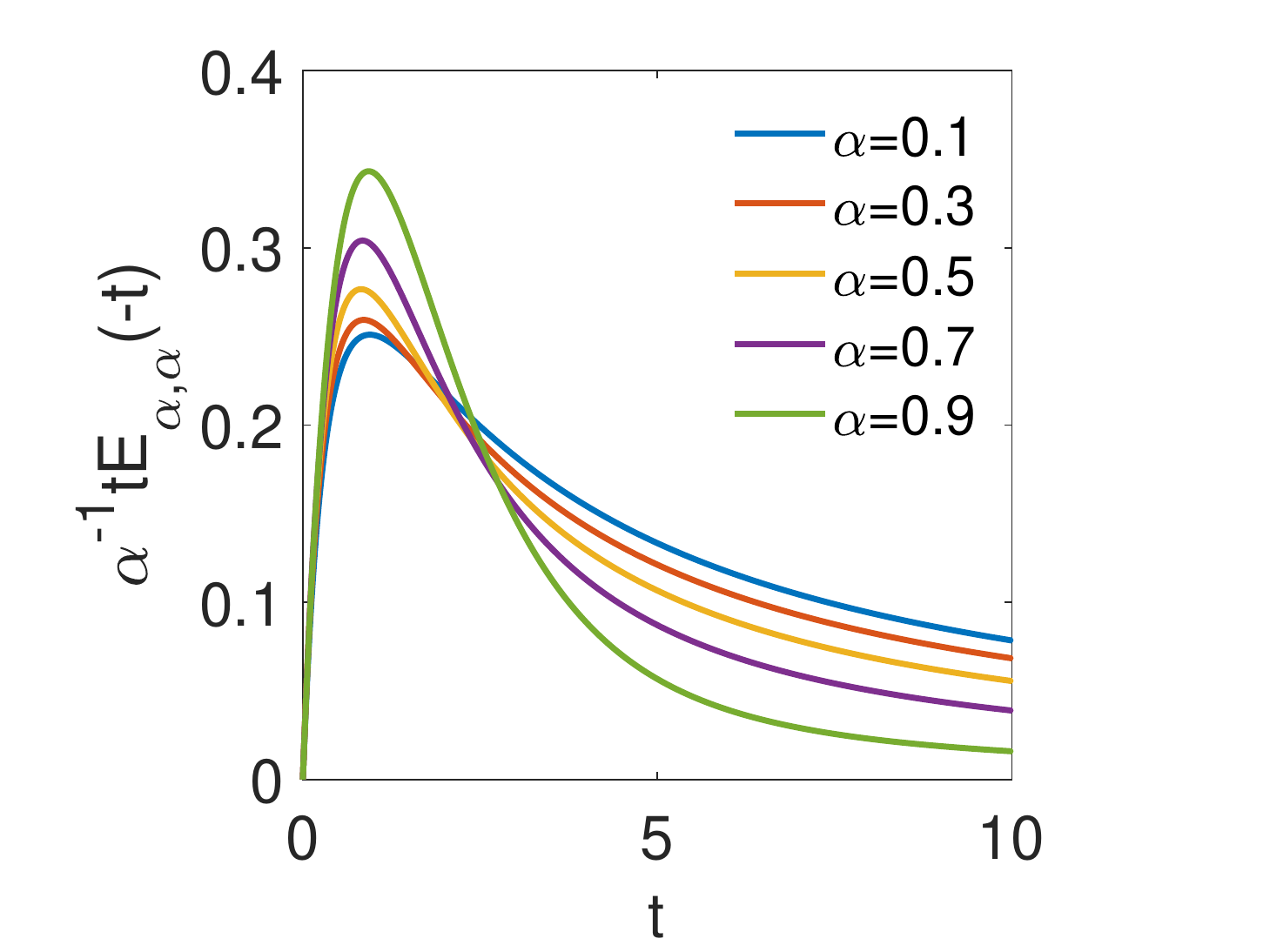} & \includegraphics[width=0.45\textwidth,trim={1.5cm 0 1cm 0.5cm}]{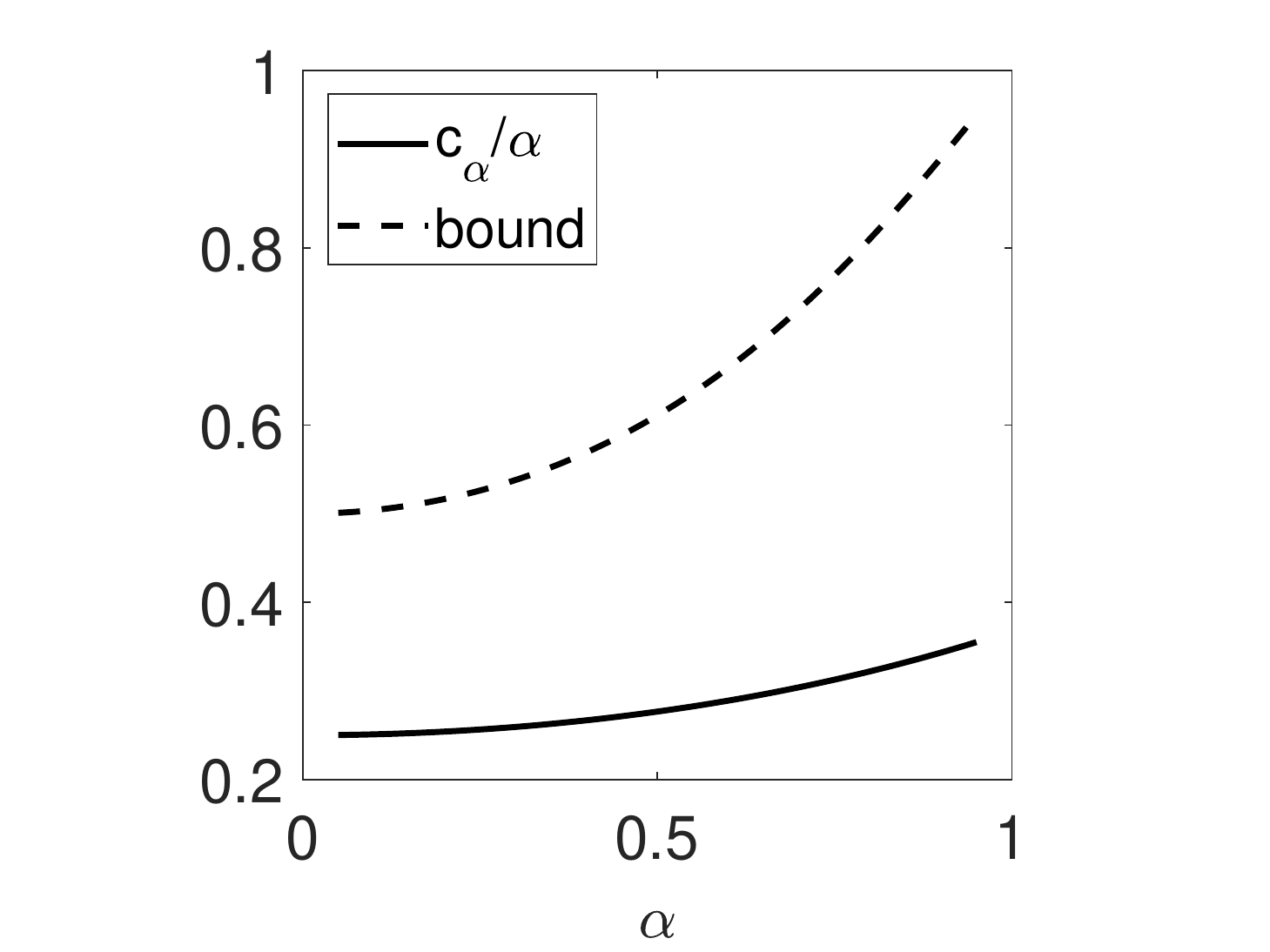}\\
    (a) $\alpha^{-1}tE_{\alpha,\alpha}(-t)$ & (b) $\frac{c_\alpha}{\alpha}$ versus $\frac{\alpha\pi}{\alpha\pi+\sin\alpha\pi}$
  \end{tabular}
  \caption{The function $\alpha^{-1}tE_{\alpha,\alpha}(-t)$ and its maximum $\frac{c_\alpha}{\alpha}$ versus the upper bound $\frac{\alpha\pi}{\alpha\pi+\sin\alpha\pi}$ in Proposition \ref{prop:bbd-Ealal}.\label{fig:cal}}
\end{figure}
\end{remark}

The next result gives the invertibility of the operator $I-P_T$ on $L^2(\omega)$.
\begin{lemma}\label{lem:PT}
Under the assumptions of Theorem \ref{thm:inverse}, there exists a $\theta>0$ depending only on $\alpha$ and $\epsilon$
such that if $\lambda_1T^\alpha<\theta$, then the operator $I-P_T$ has a bounded inverse in $B(L^2(\omega))$.
\end{lemma}
\begin{proof}
First, we bound $\|tAE(t)\|$. Using the eigenpairs $\{(\lambda_j,\varphi_j)\}_{j=1}^\infty$ of the operator $A$, we deduce
\begin{equation*}
  E(t) v = \sum_{j=1}^\infty t^{\alpha-1}E_{\alpha,\alpha}(-\lambda_jt^\alpha)(\varphi_j,v)\varphi_j,\quad \forall v\in L^2(\Omega).
\end{equation*}
Thus,
\begin{equation*}
  \|tAE(t)v\|^2 = \sum_{j=1}^\infty (\lambda_jt^\alpha E_{\alpha,\alpha}(-\lambda_jt^\alpha))^2(v,\varphi_j)^2.
\end{equation*}
Since $\sup_{t\in[0,\infty]}|tE_{\alpha,\alpha}(-t)|\leq c_\alpha<\alpha$, in view of Proposition \ref{prop:bbd-Ealal},
\begin{equation*}
  \|AE(t)v\| \leq c_\alpha t^{-1}.
\end{equation*}
Meanwhile, using the governing equation for $U(t)$, we have
\begin{equation*}
  U(t)-U(0) = {_0I_t^\alpha} \Delta U(t),
\end{equation*}
which together with the fact $\Delta U(x,t)\leq 0$  implies
\begin{align*}
 U(t)-U(T) &= ({_0I_t^\alpha}\Delta U)(t) - ({_0I_t^\alpha}\Delta U)(T)\\
            & = \frac{1}{\Gamma(\alpha)}\int_0^t[(T-s)^{\alpha-1}-(t-s)^{\alpha-1}](-\Delta U(s)) \d s+ \frac{1}{\Gamma(\alpha)}\int_t^T{(T-s)^{\alpha-1}}(-\Delta U(s))\d s.
\end{align*}
Since $(T-s)^{\alpha-1}-(t-s)^{\alpha-1}\leq 0$ and $-\Delta U(x,t)\geq 0$ in $\Omega\times[0,T]$, by Lemma \ref{lem:bdd-U}(iii)
\begin{align*}
         U(t)-U(T) & \leq \frac{1}{\Gamma(\alpha)}\int_t^T(T-s)^{\alpha-1}(-\Delta U(s))\d s \leq \frac{(T-t)^\alpha}{\Gamma(\alpha+1)}  \mu_1\lambda_1.
\end{align*}
Similarly,
\begin{align*}
  U(T)-U(t) & \leq \frac{1}{\Gamma(\alpha)}\int_0^t[(t-s)^{\alpha-1}-(T-s)^{\alpha-1}](-\Delta U(s)) \d s\\
            & \leq \frac{\mu_1\lambda_1}{\Gamma(\alpha+1)} (t^\alpha + (T-t)^\alpha - T^\alpha) \leq \frac{(T-t)^\alpha}{\Gamma(\alpha+1)}\mu_1\lambda_1.
\end{align*}
Consequently, there holds
\begin{equation*}
  \|U(s)-U(T)\|_{L^\infty(\Omega)} \leq {\frac{1}{\Gamma(\alpha+1)}}\mu_1\lambda_1(T-s)^\alpha.
\end{equation*}
Lemma \ref{lem:bdd-U}(ii) implies
\begin{equation*}
  \|u_T\|_{L^\infty(\Omega)}\leq \frac{1}{\mu_0(1-\epsilon)E_{\alpha,1}(-\lambda_1T^\alpha)}.
\end{equation*}
The preceding two estimates and Lemma \ref{lem:solop} imply
\begin{align*}
  \|P_T\|_{B(L^2(\omega))}& \leq \int_0^T\|AE(T-s)\|\|u_T\|_{L^\infty(\Omega)}\|U(s)-U(T)\|_{L^\infty(\Omega)}\d s + \|F(T)\|\\
     & \leq \int_0^T c_\alpha(T-s)^{-1}\frac{\lambda_1\mu_1}{\Gamma(\alpha+1)}(T-s)^\alpha\frac{1}{\mu_0(1-\epsilon)E_{\alpha,1}(-\lambda_1T^\alpha)}\d s + E_{\alpha,1}(-\lambda_1T^\alpha)\\
     &=\frac{c_\alpha \mu_1}{\mu_0(1-\epsilon)\alpha\Gamma(\alpha+1)E_{\alpha,1}(-\lambda_1T^\alpha)}\lambda_1T^\alpha + E_{\alpha,1}(-\lambda_1T^\alpha).
\end{align*}
Let $m(x)$ be defined by
\begin{equation*}
  m(x) = \frac{c_\alpha\mu_1}{\mu_0 (1-\epsilon)\alpha\Gamma(\alpha+1)}\frac{x}{E_{\alpha,1}(-x)} + E_{\alpha,1}(-x).
\end{equation*}
Straightforward computation shows
\begin{equation*}
  m'(x) = \frac{c_\alpha \mu_1}{\mu_0(1-\epsilon)\alpha\Gamma(\alpha+1)}\frac{E_{\alpha,1}(-x)-xE'_{\alpha,1}(-x)}{E_{\alpha,1}(-x)^2}+ E_{\alpha,1}'(-x).
\end{equation*}
Thus, $m(0)=1$ and by Proposition \ref{prop:bbd-Ealal},
\begin{equation*}
  m'(0) = \frac{c_\alpha \mu_1}{\mu_0(1-\epsilon)\alpha\Gamma(\alpha+1)}-\frac{1}{\Gamma(\alpha+1)} = {\left[\frac{c_\alpha \mu_1}{\mu_0(1-\epsilon)\alpha}-1\right]\frac{1}{\Gamma(\alpha+1)}}<0,
\end{equation*}
under the given conditions on $\epsilon,\mu_0$ and $\mu_1$ in Theorem \ref{thm:inverse}.
Thus, there exists a $\theta>0$ such that whenever $x<\theta$, $m(x)<1$, and
accordingly, for $\lambda_1T^\alpha$ sufficiently close to zero, $P_T$ is a contraction on $L^2(\omega)$. Then
by Neumann series expansion, $I-P_T$ is invertible and $(I-P_T)^{-1}$ is bounded. This completes the proof of the lemma.
\end{proof}

Now we introduce the trace operator: $\mathrm{tr}:X\to D(A)$, $v\mapsto v(T)$. Then $\mathrm{tr}\in B(X,D(A))$
and $\|\mathrm{tr}\|_{B(X,D(A))}\leq 1$. Finally, we can present the proof of Theorem \ref{thm:inverse}.
\begin{proof}
With Lemma \ref{lem:PT} at hand, the proof is identical with that of \cite{ChoulliYamamoto:1997}. We only include a
proof for the convenience of readers. We define the mapping: $K: L^2(\omega)\to L^2(\omega)$, $q\mapsto
[-Au(q)(T)]|_\omega = [-A\mathrm{tr}(I-L(q))^{-1}U]|_\omega$.
Clearly, $K$ is continuously Fr\'{e}chet differentiable, cf. Lemma \ref{lem:Lq}, and its derivative $K'$ at
$q\in L^2(\omega)$ in the direction $p$ is given by
\begin{equation*}
  K'(q)[p] = [-A\mathrm{tr}(I-L(q))^{-1}L(p)(I-L(q))^{-1}U]|_\omega.
\end{equation*}
Let $Q_T=K'(0)=[-A\mathrm{tr}L(\cdot)U]|_\omega$. Then
\begin{equation*}
  Q_T(p)=\Big[\int_0^T-AE(T-s)p[U(s)-U(T)]\d s + (F(T)-I)pU(T)\Big]\Big|\big._\omega.
\end{equation*}
We define a multiplication operator $M:L^2(\omega)\to L^2(\omega)$, $p\to U(T)p$. Then $M$
is invertible, and its inverse is exactly the multiplication operator by $u_T$.
Consequently, $Q_TM^{-1}=P_T-I$. By Lemma \ref{lem:PT}, $(P_T-I)^{-1}$ belongs
to $B(L^2(\omega))$. Therefore, $Q_T$ has a bounded inverse and $Q_T^{-1}=M^{-1}(P_T-I)^{-1}$.
By the implicit function theorem, $K$ is locally a $C^1$-diffeomorphism from
a neighborhood of $0$ onto a neighborhood of $K(0)$. In particular, $K^{-1}$ is
Lipschitz continuous in a neighborhood of $K(0)$. Then Theorem \ref{thm:inverse}
follows by noting the following inequality
\begin{equation*}
  \|Au(q_1)(T)|_\omega-Au(q_2)(T)|_\omega\|_{L^2(\omega)}\leq \|u(q_1)(T)-u(q_2)(T)\|_{D(A)},
\end{equation*}
for any $q_1,q_2\in L^2(\omega)$.
\end{proof}

\section{Fixed point algorithm}\label{sec:alg}

Now we propose a simple fixed point algorithm to find the potential $q$ from the terminal observation.
Given a noisy version of the exact data $g=u(q^\dag)(T)$ corresponding to the exact potential $q^\dag$
and an initial guess $q^0$, we employ the following fixed point iteration
\begin{equation}\label{eqn:fp}
   q^{k+1} = F(q^k), \quad \mbox{with }F(q) = q + \lambda A^{-1}(u(q)(T)-g),
\end{equation}
where $\lambda>0$ is a relaxation parameter and $A=-\Delta$ is the negative Laplacian with a zero Dirichlet
boundary condition. In the absence of the preconditioning operator $A^{-1}$, the iteration \eqref{eqn:fp}
was proposed in \cite{Rundell:1987} for the standard parabolic problem. For both normal diffusion and
subdiffusion, the unpreconditioned version works very robustly for exact data, but it tends to suffer from severe
numerical instability in the presence of data noise. This is attributed to the fact that the noise in the data
$g$ is amplified by a factor $\lambda$ at each iteration, in view of the smoothing property of the
solution operator, and the noise effect accumulates very rapidly so as to completely spoil the reconstruction after a few iterations. The preconditioner $A^{-1}$ is to mitigate the
deleterious effect of noise in the observation $g$ by implicitly filtering out the high-frequency components
present in the noise thereby achieving a form of regularization \cite{EnglHankeNeubauer:1996}.
Numerically, the scheme is straightforward to implement since it requires only one forward solve,
and the preconditioning step incurs very little extra computational effort.

We have the following contractive property on the cone $S=\{h\in C(\overline{\Omega}): h(x)\geq 0\}$.
\begin{proposition}
For any nonnegative $u_0\not\equiv0$ and any $q\in S$,  the linearized map $F'$ is contractive on $S$ in the following sense
\begin{equation*}
  \|F'(q)h\|_{L^2(\Omega)}<\|h\|_{L^2(\Omega)},\quad \forall h\in S,
\end{equation*}
provided that the relaxation parameter $\lambda$ is sufficiently small.
\end{proposition}
\begin{proof}
For any $q,h\in S$, the G\^{a}teaux derivative $F$ is given by
\begin{equation*}
  F'(q)[h] = h -\lambda A^{-1} v(T),
\end{equation*}
where $v\equiv v(q,h)$ satisfies the following inhomogeneous problem
\begin{align*}
 \left\{\begin{aligned}
    \partial_t^\alpha v & = \Delta v -qv + hu(q),\quad \mbox{in } \Omega\times(0,T],\\
     v & = 0,\quad \mbox{on }\partial\Omega\times(0,T],\\
     v(0)& = 0,\quad\mbox{in }\Omega.
 \end{aligned}\right.
\end{align*}
By the ``strong'' maximum principle for the subdiffusion model \cite{LuchkoYamamoto:2017} and the nonnegativity of $u_0$ and $q$ that
\begin{equation*}
  0<u(q)(x,t)\leq \|u_0\|_{L^\infty(\Omega)}\quad \forall (x,t)\in \Omega\times(0,T].
\end{equation*}
Since $h\in S$, the maximum principle \cite{LuchkoYamamoto:2017} shows $ v(x,t)\geq$ in $\Omega\times[0,T]$.
Further, with $f(t)=hu(t)$, the solution $v$ can be represented by
\begin{equation*}
  v(t) = \int_0^tE(t-s)f(s)\d s = \sum_{j=1}^\infty \int_0^t(t-s)^{\alpha-1}E_{\alpha,\alpha}(-\lambda_j(t-s)^\alpha)(f(s),\varphi_j)\d s\varphi_j.
\end{equation*}
In particular,
\begin{equation*}
  (v(T),\varphi_1) = \int_0^t(t-s)^{\alpha-1}E_{\alpha,\alpha}(-\lambda_1(t-s)^\alpha)(f(s),\varphi_1)\d s.
\end{equation*}
Now if $h\not\equiv0$, then for any fixed $t>0$, $\mathrm{supp}(f(t))
=\mathrm{supp}(h)$, and by the positivity of $\varphi_1$ in $\Omega$, $(f(t),\varphi_1)>0$
for any $t>0$. Thus, $0\leq v(T)\not\equiv0$, and by the properties of elliptic problems,
$A^{-1}v(T)>0$ in $\Omega$. Thus, by choosing $\lambda$ sufficiently small (depending on $h$), we
deduce the desired assertion.
\end{proof}

The next result shows that the fixed point iteration \eqref{eqn:fp} can actually also be interpreted as a preconditioned
gradient descent method, under certain restrictions on $u_0$ and the residual $u(q)(T)-g$. The
descent property can be numerically observed in a more general case, which, however, remains to be proved.
\begin{proposition}\label{prop:descent}
If $u_0\not\equiv0$ is nonnegative and $u(q)(T)-g\not\equiv0$ is not sign changing, then $A^{-1}(u(q)(T)-g)$ is a descent
direction to the functional $J(q)=\tfrac12\|u(q)(T)-g\|_{L^2(\Omega)}^2$.
\end{proposition}
\begin{proof}
Let $w\equiv w(q)$ solve the adjoint problem
\begin{equation*}
 \left\{\begin{aligned}
    _{t}^{ \kern -.1em R} \kern -.2em \partial^{\alpha}_{\kern -.1em T} w & = \Delta w -qw,\quad \mbox{in } \Omega\times(0,T],\\
     w & = 0,\quad \mbox{on }\partial\Omega\times(0,T],\\
     _{t}^{ \kern -.1em R} \kern -.2em \partial^{\alpha-1}_{\kern -.1em T}w(T)& = u(q)(T)-g,\quad\mbox{in }\Omega,
 \end{aligned}\right.
\end{equation*}
where the notation $_{t}^{ \kern -.1em R} \kern -.2em \partial^{\alpha}_{\kern -.1em T} w$ denotes the right-sided
Riemann-Liouville fractional derivative (based at $T$), and $_{t}^{ \kern -.1em R} \kern -.2em \partial^{\alpha-1}_{\kern -.1em T}w(T) $
the Riemann-Liouville fractional integral of order $1-\alpha$. Further, using the solution operator $E_q$ associated with the
$\Delta -q$, $w$ can be represented by
\begin{equation}\label{eqn:solrep-w}
  w(t) = E_q(T-t)(u(q)(T)-g) = \int_t^T E_q(s-t)(u(q)(T)-g)\delta_T(s)\d s,
\end{equation}
where $\delta_T(s)$ denotes the Dirac delta function at $T$.
Then direct computation shows that the gradient $J'(q)$ to the functional $J(q)$ is given by
\begin{equation*}
  J'(q) = - \int_0^Tu(q) w(q)\d t.
\end{equation*}
Now it follows that
\begin{align*}
  \int_\Omega J'(q) A^{-1}(u(q)(T)-g)\d x & = - \int_0^T\int_\Omega u(q)w(q)A^{-1}(u(q)(T)-g)\d x\d t .
\end{align*}
By the maximum principle for elliptic problems, $u(q)(T)-g\lessgtr0$ in $\Omega$ implies $A^{-1}(u(q)(T)-g)
\lessgtr0$ in $\Omega$, and similarly, $u(q)$ is positive almost everywhere in $\Omega\times (0,T)$ for
nonnegative $u_0\not\equiv0$ \cite{LuchkoYamamoto:2017}. Meanwhile, in view of the representation \eqref{eqn:solrep-w}, using
a density argument (i.e., approximating the singular source $(u(q)(T)-g)\delta_T(s)$ with $(u(q)(T)-g)\phi_n(t)$,
with $\phi_n\geq 0$ being smooth and $\phi_n(t)\to \delta_T(t)$ weakly; see \cite{LuchkoYamamoto:2017} for relevant
argument) and the weak maximum principle for subdiffusion, $u(q)(T)-g\lessgtr 0$ implies
$w(t)\lessgtr 0 $ almost everywhere in $\Omega\times(0,T)$. Consequently, we arrive at
\begin{equation*}
  \int_\Omega J'(q) A^{-1}(u(q)(T)-g)\d x<0,
\end{equation*}
i.e., $A^{-1}(u(q)(T)-g)$ is a descent direction to the functional $J(q)$.
\end{proof}

\begin{algorithm}[hbt!]
  \center
  \caption{Anderson acceleration for the fixed point iteration \eqref{eqn:fp}.\label{alg:aa}}
  \begin{algorithmic}[1]
    \STATE Give the initial guess $q^0$, memory parameter $m\geq1$, and the maximum iteration number $K$.
    \FOR {$k=0,1,\ldots,K$}
    \STATE Set $m_k=\min(m,k)$.
    \STATE Compute $R_k=[r_{k-m_k},\ldots r_{k}]$, with $r_i=F(q^i)-q^i$.
    \STATE Find $\beta^{(k)}\in \mathbb{R}^{m_k+1}$ by
      \begin{equation*}
        \beta^{(k)}\in \arg\min_{\beta\in \mathbb{R}^{m_k+1},\sum_{i=0}^{m_k}\beta_k=1}\|R_k\beta\|.
      \end{equation*}
    \STATE Set $q^{k+1}=\sum_{i=0}^{m_k}\beta_i^{(k)}F(q^i)$.
    \STATE Check the stopping criterion.
    \ENDFOR
  \end{algorithmic}
\end{algorithm}

Numerical experiments indicate that the convergence behavior of fixed point iteration \eqref{eqn:fp} depends
very much on the relaxation parameter $\lambda>0$: if $\lambda$ is small, then it converges steadily but only
slowly, whereas for large $\lambda$, the convergence may be unstable and suffers from large oscillations. In
order to accelerate the convergence, we employ the classical Andersson acceleration technique \cite{Anderson:1965},
which can be viewed as a version of GMRES for nonlinear problems \cite{WalkerNi:2011}; see the review
\cite{BrezinskiSaad:2018} for other related extrapolation techniques. The complete procedure for Anderson
acceleration is listed in Algorithm \ref{alg:aa}. The integer $m$ controls the number of memory terms used
for the Anderson update. Thus, the acceleration step only involves simple algebraic manipulations, and the
associated computational overhead is negligible. In our experiment below, $m=2$ represents a good choice. At
line 7 of the algorithm, the stopping criterion of the iteration can employ the standard discrepancy principle, i.e.,
\begin{equation}\label{eqn:dp}
  k^* = \min_{k \geq 1}\{ \|u(q^k)-g\|_{L^2(\Omega)} \leq \tau\delta\},
\end{equation}
where $\tau>1$ is the tolerance, and $\delta=\|g-u(q^\dag)\|_{L^2(\Omega)}$ is the noise level.
The discrepancy principle is a well established early stopping strategy for iterative regularization methods
\cite{EnglHankeNeubauer:1996}. The fixed point algorithm and its accelerated variant exhibit a very
similar behavior in practice, when noise is present in the data; see Section \ref{sec:numer} for numerical illustrations.

Despite the enormous empirical success, the global convergence of Anderson acceleration remains completely open, even for affine
linear maps with fixed memory (the case of linear map with full memory is well known due to its connection with
GMRES \cite{WalkerNi:2011}). The local convergence of Anderson acceleration for contractive maps was studied recently in
\cite{TothKelley:2015,EvansPollock:2020}. However, these results do not apply to the inverse potential
problem, since the associated map is not a contraction.

\begin{remark}
In the fixed point iteration \eqref{eqn:fp}, the update does not change the boundary condition
of the initial guess $q^0$. Thus, it is implicitly assumed that the boundary condition is
exactly known. Further, for $g\in L^2(\Omega)$, by the standard elliptic regularity result,
the update increment $A^{-1}(u(q)(T)-g)$ belongs to $H^2(\Omega)$, and thus the regularity of
the initial guess $q^0$ essentially determines the regularity of the iterates, and the
algorithm is most suitable for recovering a smooth potential.
\end{remark}

\begin{remark}
There are alternative choices of fixed point algorithms. One popular choice
is due to Isakov \cite{Isakov:1991}: given the initial guess $q^0$, it reads
\begin{equation*}
  q^{k+1} = \frac{\partial_t^\alpha u(q^k)(T)-\Delta u(q^k)(T)}{g}.
\end{equation*}
The convergence of the algorithm in the time-fractional case has been analyzed
in \cite{ZhangZhou:2017}, provided that the terminal time $T$ is sufficiently large.
Anderson acceleration might also be used to accelerate this algorithm.
\end{remark}

\section{Numerical reconstructions and discussions}\label{sec:numer}

Now we illustrate the accuracy and efficiency of the fixed point algorithm
\eqref{eqn:fp} with one- and two-dimensional numerical examples. The direct problem is solved by a fully discrete
scheme based on the Galerkin finite element method in space and backward Euler convolution quadrature
in time, which is first-order accurate in time and second-order accurate in space \cite{JinLazarovZhou:2016};
(see \cite{JinLazarovZhou:2019} for an overview of existing schemes). The noisy data $g$ is generated by
\begin{equation*}
  g(x) = u(q^\dag)(x,T) + \epsilon\sup_{x\in\Omega}|u(q^\dag)(x,T)|\xi(x),\quad x\in \Omega,
\end{equation*}
where the noise $\xi(x)$ follows the standard Gaussian distribution, and $\epsilon\geq 0$ denotes the (relative) noise level.
The exact data $u(q^\dag)(x,T)$ is generated using a finer spatial-temporal mesh in order to avoid the inverse crime. In Anderson acceleration, the memory parameter
$m$ is fixed at 2, and the relaxation parameter $\lambda$ is fixed at 1000 and 100 for one- and two-dimensional problems,
respectively. Note that this choice of $\lambda$ is not optimized, since the optimal choice depends strongly on the problem data, e.g., $T$ and $u_0$. Nonetheless,
the numerical experiments below indicate that Anderson acceleration is fairly robust with respect to $\lambda$, and works for a broad range
of $\lambda$ values. Throughout, the parameter $\tau$ in the discrepancy principle \eqref{eqn:dp} is fixed at $\tau=1.01$.
Below, For a given reconstruction $q^*$, we compute two metrics, the $L^2(\Omega)$-error $e_q$ and the residual
$r_q$, defined, respectively, by
\begin{equation*}
  e_q = \|q^\dag - q^*\|_{L^2(\Omega)} \quad \mbox{and}\quad r_q=\|u(q^*)-g\|_{L^2(\Omega)},
\end{equation*}
where $q^\dag$ denotes the exact potential. Unless otherwise specified, the results presented below are obtained by the fixed point
algorithm \eqref{eqn:fp} with Anderson acceleration, with a zero initial guess.

\subsection{Results for the one-dimensional case}
First we present two one-dimensional examples on the unit interval $\Omega=(0,1)$. In the computation, the domain
$\Omega$ is divided into $M$ equal subintervals, and the time interval $(0,T)$ is divided into $N$ subintervals. To
generate data, we take $M=1000$ and $N=1000$, whereas for the inversion, $M=200$ and $N=500$. The fixed point iteration
\eqref{eqn:fp} is run for at most 1000 iterations.

The first example is to recover a smooth potential.
\begin{example}\label{exam:1d1}
$u_0=\sin\pi x + \frac{1}{100}x(1-x)$ and $q^\dag(x) = e^x\sin(2\pi x)$.
\end{example}

Note that the initial condition $u_0$ is chosen to fulfill the conditions in Theorem \ref{thm:inverse}.
The numerical results for Example \ref{exam:1d1} are shown in Tables \ref{tab:exam1T1aa}--\ref{tab:exam1T001aa},
with three different final times,  $T=0.01$, $T=0.1$ and $T=1$, which also include the results for normal
diffusion (i.e., $\alpha=1.00$). In the tables, the numbers refer to the
reconstruction error $e_q$, and the numbers in the brackets denote the stopping index determined by the
discrepancy principle \eqref{eqn:dp}. It is observed that the error $e_q$ decreases
steadily as the noise level $\epsilon$ tends to zero for all three fractional orders $\alpha$ and final time $T$.
For each fixed $T$ and $\epsilon$, the accuracy does not change much with respect to $\alpha$, and thus the
fractional order $\alpha$ does not influence much the behavior of the reconstruction error. Nonetheless, for any fixed $\alpha$,
when the data is noise free, the error $e_q$ increases with the time $T$, although only very slightly. These
observations are consistent with the local Lipschitz stability in Theorem \ref{thm:inverse} (and the stability
for the parabolic case \cite{ChoulliYamamoto:1997}), which holds for all $\alpha\in (0,1]$,
so long as the terminal time $T$ is sufficiently small. The numerical experiments actually indicate that even for much
large $T$, the inverse problem exhibits nearly identical behavior in terms of the reconstruction error $e_q$,
indicating similar degree of ill-posedness. See Fig. \ref{fig:exam1:recon} for exemplary reconstructions for
Example \ref{exam:1d1} with $T=1$ at two noise levels. The reconstructions are
largely comparable with each other for different fractional orders, corroborating Table \ref{tab:exam1T1aa}.
However, the last observation for large $T$ seems no longer valid for
normal diffusion (i.e., $\alpha=1$), for which the numerical reconstruction becomes much more challenging;
the fixed point algorithm does not work as well as in the fractional case: it takes many more iterations to reach
the discrepancy principle, and yet the reconstruction is generally inferior at all noise levels. This agree
also with the empirical observations in the last column of Fig. 1 of \cite{KaltenbacherRundell:2019}.

\begin{table}[hbt!]
  \centering
  \caption{The reconstruction error $e_q$ for Example \ref{exam:1d1} with $T=0.01$.\label{tab:exam1T001aa}}
  \begin{tabular}{cccccc}
  \toprule
  $\alpha\backslash \epsilon$ & 0 & 1e-3 & 5e-3 & 1e-2 & 5e-2 \\
  \midrule
  0.25 &  2.21e-3 (1000) &  4.77e-2 ( 8) &  1.36e-1 (7) &  2.04e-1 (4) &  3.91e-1 (3)\\
  0.50 &  2.71e-3 (1000) &  4.43e-2 (15) &  1.41e-1 (7) &  1.99e-1 (4) &  4.06e-1 (3)\\
  0.75 &  2.01e-3 (1000) &  4.91e-2 ( 7) &  9.06e-2 (6) &  2.02e-1 (3) &  1.04e0  (1)\\
  1.00 &  3.36e-3 (1000) &  7.85e-2 ( 6) &  2.12e-1 (3) &  2.83e-1 (3) &  1.00e0  (1)\\
  \bottomrule
  \end{tabular}
\end{table}

\begin{table}[hbt!]
  \centering
  \caption{The reconstruction error $e_q$ for Example \ref{exam:1d1} with $T=0.1$.\label{tab:exam1T01aa}}
  \begin{tabular}{cccccc}
  \toprule
  $\alpha\backslash \epsilon$ & 0 & 1e-3 & 5e-3 & 1e-2 & 5e-2 \\
  \midrule
  0.25 & 2.89e-3 (1000) &  4.69e-2 (9) &   8.06e-2 (9) &  2.04e-1 (4) &  3.88e-1 (3)\\
  0.50 & 3.65e-3 (1000) &  4.96e-2 (8) &   8.10e-2 (7) &  2.04e-1 (4) &  3.78e-1 (3)\\
  0.75 & 5.44e-3 (1000) &  4.69e-2 (9) &   7.95e-2 (8) &  2.03e-1 (4) &  3.50e-1 (3)\\
  1.00 & 4.99e-3 (1000) &  4.59e-2 (9) &   7.08e-2 (8) &  2.01e-1 (4) &  3.19e-1 (3)\\
  \bottomrule
  \end{tabular}
\end{table}

\begin{table}[hbt!]
  \centering
  \caption{The reconstruction error $e_q$ for Example \ref{exam:1d1} with $T=1$.\label{tab:exam1T1aa}}
  \begin{tabular}{cccccc}
  \toprule
  $\alpha\backslash \epsilon$ & 0 & 1e-3 & 5e-3 & 1e-2 & 5e-2 \\
  \midrule
  0.25 & 3.69e-3 (1000) & 4.56e-2 (16)   & 9.55e-2 (10) & 2.03e-1 (4) & 3.99e-1 (3)\\
  0.50 & 5.14e-3 (1000) & 4.74e-2 (11)   & 1.32e-1 ( 9) & 2.02e-1 (4) & 4.10e-1 (3)\\
  0.75 & 1.19e-2 (1000) & 4.93e-2 (12)   & 7.85e-2 (10) & 1.96e-1 (4) & 3.54e-1 (4)\\
  1.00 & 2.12e-1 (1000) & 2.12e-1 (1000) & 2.13e-1 (1000)&  2.17e-1(46) & 3.14e-1 (43)\\
   \bottomrule
  \end{tabular}
\end{table}

\begin{figure}[hbt!]
  \centering
  \begin{tabular}{cc}
  \includegraphics[width=0.45\textwidth,trim={1.5cm 0 1cm 0.0cm}]{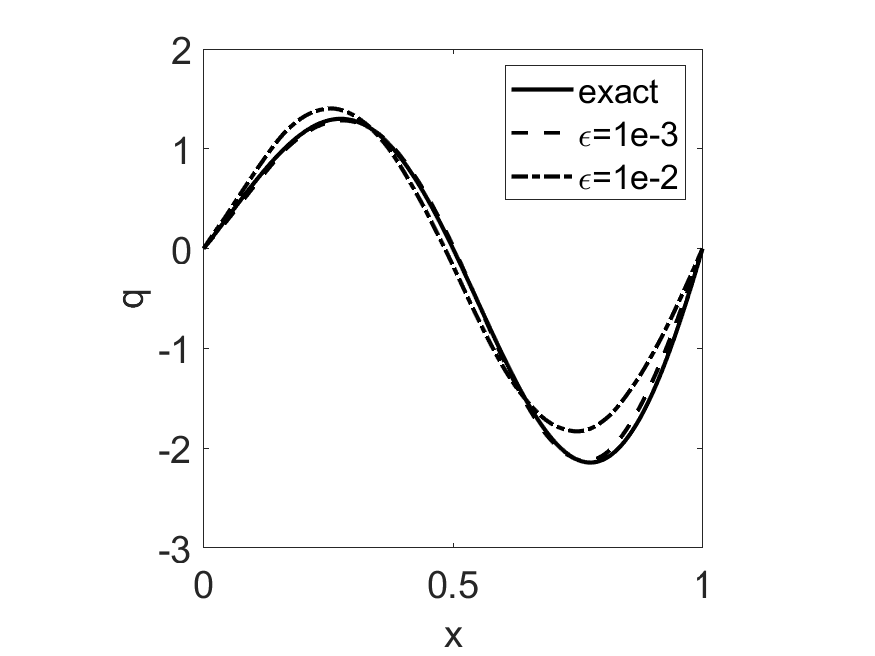} & \includegraphics[width=0.45\textwidth,trim={1.5cm 0 1cm 0.0cm}]{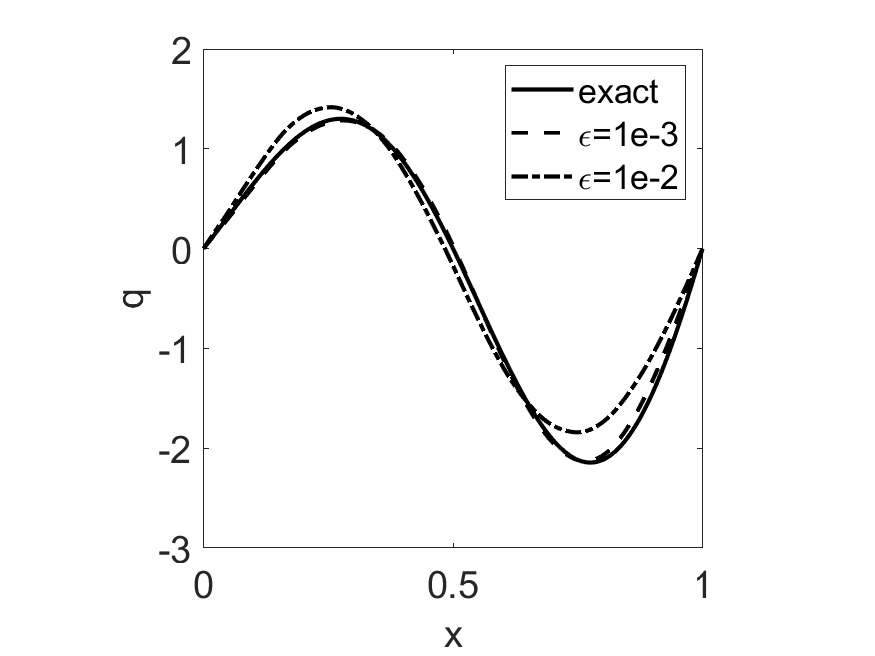}\\
  (a) $\alpha=0.25$ & (b) $\alpha=0.50$\\
  \includegraphics[width=0.45\textwidth,trim={1.5cm 0 1cm 0.0cm}]{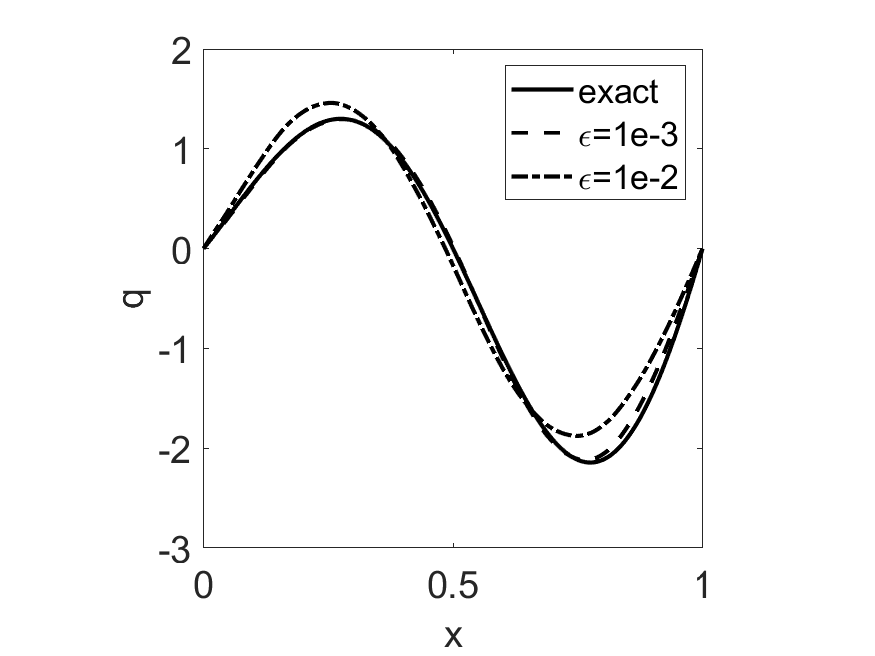} & \includegraphics[width=0.45\textwidth,trim={1.5cm 0 1cm 0.0cm}]{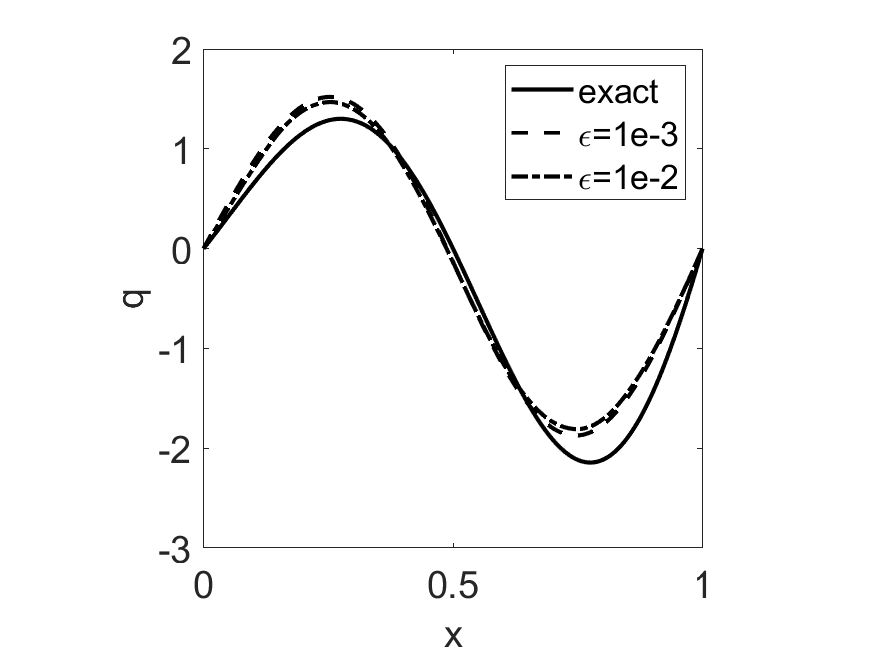}\\
  (c) $\alpha=0.75$ & (d) $\alpha = 1.00$
  \end{tabular}
  \caption{Numerical reconstructions for Example \ref{exam:1d1} at $T=1$ with different $\alpha$ values.\label{fig:exam1:recon}}
\end{figure}

Tables \ref{tab:exam1T001aa}--\ref{tab:exam1T1aa} indicate that with Anderson acceleration and discrepancy principle,
the fixed point algorithm is generally terminated after about 10 iterations for low noise level, and 5 iterations for
high noise levels. In contrast, the fixed point algorithm \eqref{eqn:fp} takes far more iterations, by a factor of 10;
see Table \ref{tab:exam1} for related results for Example \ref{exam:1d1} with $T=1$. Nonetheless, with or without
acceleration, the obtained reconstruction errors are largely comparable with each other, except the case $\epsilon=0$,
for which the iteration \eqref{eqn:fp} requires far more than 1000 iterations in order to achieve comparable accuracy
with that in Table \ref{tab:exam1T1aa}. Thus, Anderson acceleration is very effective in speeding up the convergence,
while maintaining comparable accuracy. It is worth noting that for $T=1$, the results for normal diffusion are inferior
for $T=1$, as manifested by the fact that the convergence of the fixed point algorithm suffers seriously and the least-squares
problem in Anderson acceleration exhibits pronounced ill-conditioing, which necessitates proper regularization (done via SVD here). Also the
accelerating effect of Anderson acceleration is less dramatic, although it does converge after more iterations, when
compared with that for smaller $T$ or small $\alpha$. These observations seem to indicate the dramatic difference in the behavior
of the inverse potential problem for subdiffusion and normal diffusion at large time $T$, and the fractional case is far more amenable with
numerical reconstruction.

\begin{table}[hbt!]
  \centering
  \caption{The reconstruction error $e_q$ for Example \ref{exam:1d1} with $T=1$, without Anderson acceleration.\label{tab:exam1}}
  \begin{tabular}{cccccc}
  \toprule
  $\alpha\backslash \epsilon$ & 0 & 1e-3 & 5e-3 & 1e-2 & 5e-2 \\
  \midrule
  0.25 & 2.49e-2 (1000) &  4.93e-2 (368) &  1.43e-1 (110) &  2.19e-1 ( 75) &  8.23e-1 (18)\\
  0.50 & 3.03e-2 (1000) &  4.96e-2 (496) &  1.43e-1 (148) &  2.17e-1 (101) &  8.10e-1 (25)\\
  0.75 & 4.64e-2 (1000) &  5.02e-2 (896) &  1.42e-1 (267) &  2.15e-1 (182) &  7.97e-1 (46)\\
  1.00 & 1.18e0  (1000) &  1.18e0  (1000)&  1.18e0  (1000)&  1.19e0 (1000) &  1.19e0  (1000)\\
  \bottomrule
  \end{tabular}
\end{table}

The convergence behavior of the acceleration scheme is shown in Fig. \ref{fig:exam1:conv-aa}. Note that
the reconstruction error $e_q$ first decreases, and then starts to increase as the iteration further
proceeds. This behavior is very similar to semi-convergence typically observed for an iterative regularization method (e.g.,
Landweber iteration). The discrepancy principle \eqref{eqn:dp} can choose a suitable stopping index before the
divergence kicks in, indicated by the red circle in the plots, and the attained reconstruction error is only
slightly larger than the optimal value (along the trajectory), showing the optimality of the discrepancy principle.
Further, a few extra iterations beyond the stopping index does not greatly deteriorate the reconstruction,
i.e., the algorithm enjoys excellent numerical stability. This highly desirable property is attributed to the use of the
preconditioner $A^{-1}$ in the iteration \eqref{eqn:fp}. Surprisingly, the residual $r_q$ is monotonically
decreasing as the iteration proceeds, and eventually levels off at the noise level $\delta$. That is, the
fixed point iteration is actually a descent method for minimizing the residual $r_q$, an interesting fact that
remains to be rigorously established (see Proposition \ref{prop:descent} for a partial justification). Thus,
overall, the algorithm with discrepancy principle is an effective reconstruction method.

\begin{figure}[hbt!]
  \centering
  \begin{tabular}{cc}
    \includegraphics[width=0.45\textwidth,trim={1.5cm 0 1cm 0.5cm}]{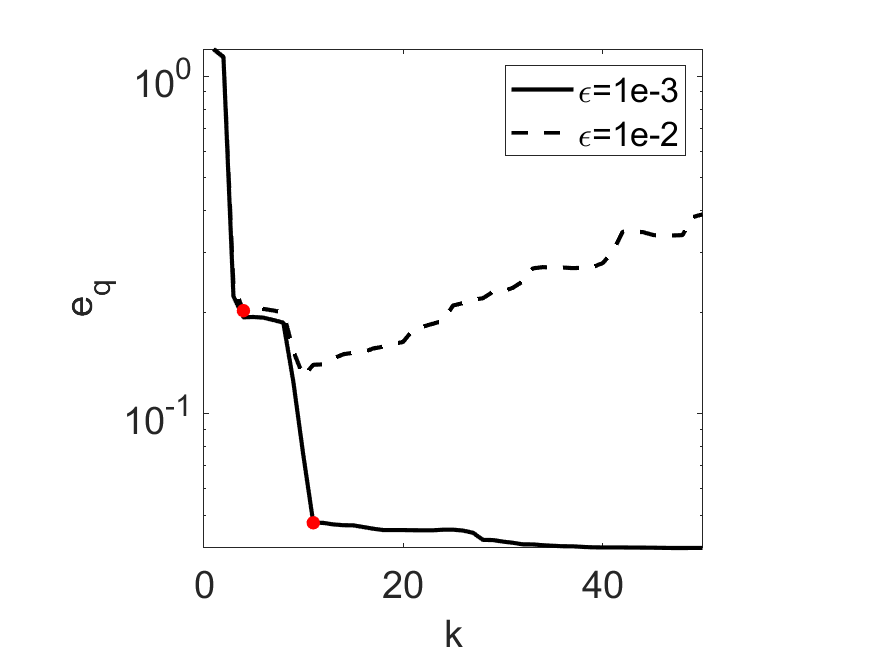} & \includegraphics[width=0.45\textwidth,trim={1.5cm 0 1cm 0.5cm}]{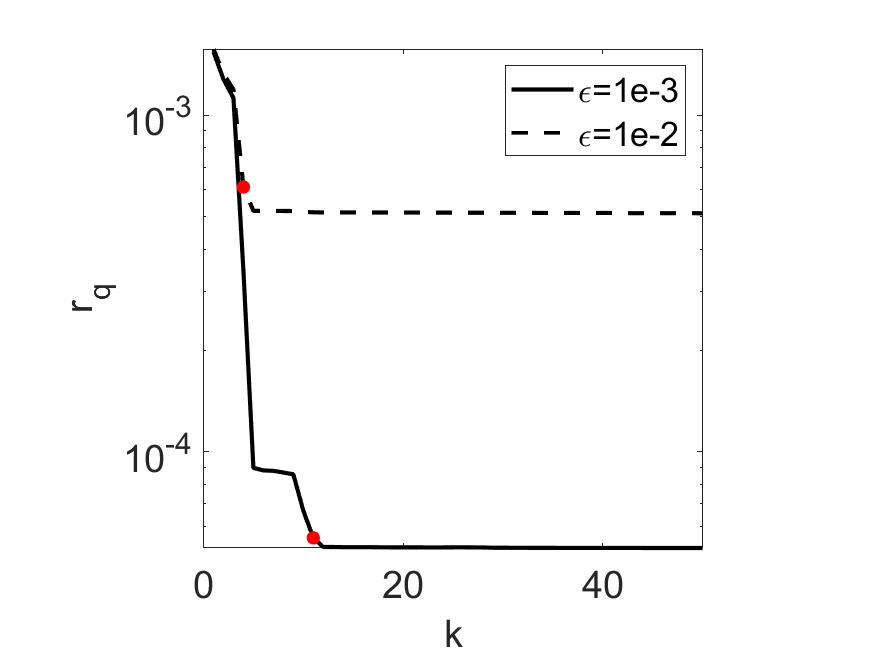}\\
    (a) error $e_q$ & (b) residual $r_q$
  \end{tabular}
  \caption{Convergence behavior of the fixed point algorithm with Anderson acceleration for Example \ref{exam:1d1} with $\alpha=0.5$ at $T=1$.
  In the plots, the red circle refers to the stopping index determined by the discrepancy principle \eqref{eqn:fp}. \label{fig:exam1:conv-aa}}
\end{figure}

The next example is about recovering a nonsmooth coefficient.
\begin{example}\label{exam:1d2}
$u_0(x)=1+\frac{3}{2}\sin2\pi x$ and $q^\dag(x) = \min(x,1-x)$.
\end{example}

Note that the given initial condition $u_0$ does not satisfy the condition of Theorem
\ref{thm:inverse}, since it does not satisfy the required regularity condition and also
changes sign in the domain, and the true potential $q^\dag$ is also less smooth. The numerical results
for Example \ref{exam:1d2} are summarized in Tables \ref{tab:exam2T001aa}--\ref{tab:exam2T1aa}.
Similar to Example \ref{exam:1d1}, it is observed that the fractional order $\alpha$ and
the terminal time $T$ does not affect much the reconstruction accuracy, indicating generic
ill-posedness of the inverse problem, irrespective of the fractional order $\alpha$.
See Fig. \ref{fig:exam2:recon} for numerical reconstructions for the case $T=1$; and like before, the results for the case
$\alpha=1.00$ are inferior to that in the fractional case. Overall, the reconstructions represent acceptable
approximations. Unsurprisingly, the approximation error is largely around the kink, where the exact
potential $q^\dag$ exhibits weak singularity. This is attributed to the smoothing effect of
the preconditioner $A^{-1}$ in the fixed point update. Thus, the iterates are overly
smooth when compared with the exact one $q^\dag$.

\begin{table}[hbt!]
  \centering
  \caption{The reconstruction error $e_q$ for Example \ref{exam:1d2} with $T=0.01$.\label{tab:exam2T001aa}}
  \begin{tabular}{cccccc}
  \toprule
  $\alpha\backslash \epsilon$ & 0 & 1e-3 & 5e-3 & 1e-2 & 5e-2 \\
  \midrule
  0.25 & 3.53e-3 (1000) &  2.68e-2 (5) &  3.70e-2 (2) &  3.95e-2 (2) &  7.59e-2 (2)\\
  0.50 & 7.48e-3 (1000) &  2.56e-2 (5) &  4.02e-2 (2) &  4.54e-2 (2) &  1.13e-1 (2)\\
  0.75 & 1.63e-2 (1000) &  2.57e-2 (4) &  3.81e-2 (2) &  4.99e-2 (2) &  1.79e-1 (2)\\
  1.00 & 1.18e-1 (1000) &  1.46e-1 (8) &  1.35e-1 (1) &  1.16e-1 (1) &  1.18e-1 (1)\\
  \bottomrule
  \end{tabular}
\end{table}

\begin{table}[hbt!]
  \centering
  \caption{The reconstruction error $e_q$ for Example \ref{exam:1d2} with $T=0.1$.\label{tab:exam2T01aa}}
  \begin{tabular}{cccccc}
  \toprule
  $\alpha\backslash \epsilon$ & 0 & 1e-3 & 5e-3 & 1e-2 & 5e-2 \\
  \midrule
  0.25 & 3.25e-3 (1000) &  1.95e-2 (6) &  3.53e-2  (2) &  3.69e-2 (2) &  1.11e-1 (1)\\
  0.50 & 7.58e-3 (1000) &  2.79e-2 (5) &  3.76e-2  (2) &  3.96e-2 (2) &  6.91e-2 (2)\\
  0.75 & 9.57e-3 (1000) &  2.83e-2 (5) &  4.45e-2  (4) &  4.55e-2 (2) &  7.91e-2 (2)\\
  1.00 & 1.94e-2 (1000) &  1.87e-2 (5) &  3.67e-2  (2) &  3.85e-2 (2) &  6.91e-2 (2)\\
  \bottomrule
  \end{tabular}
\end{table}

\begin{table}[hbt!]
  \centering
  \caption{The reconstruction error $e_q$ for Example \ref{exam:1d2} with $T=1$.\label{tab:exam2T1aa}}
  \begin{tabular}{cccccc}
  \toprule
  $\alpha\backslash \epsilon$ & 0 & 1e-3 & 5e-3 & 1e-2 & 5e-2 \\
  \midrule
  0.25 & 3.21e-3 (1000) &  1.69e-2 (6) &  3.43e-2 (2) &  3.53e-2 (2) &  5.66e-2 (1)\\
  0.50 & 6.52e-3 (1000) &  2.79e-2 (7) &  3.44e-2 (2) &  3.52e-2 (2) &  4.99e-2 (2)\\
  0.75 & 1.06e-2 (1000) &  3.04e-2 (6) &  4.36e-2 (5) &  3.50e-2 (2) &  4.57e-2 (2)\\
  1.00 & 6.02e-2 (1000) &  6.02e-2 (1000)&6.04e-2 (5) &  6.06e-2 (5) &  6.06e-2 (4)\\
  \bottomrule
  \end{tabular}
\end{table}

\begin{figure}[hbt!]
  \centering
  \begin{tabular}{cc}
  \includegraphics[width=0.45\textwidth,trim={1.5cm 0 1cm 0.0cm}]{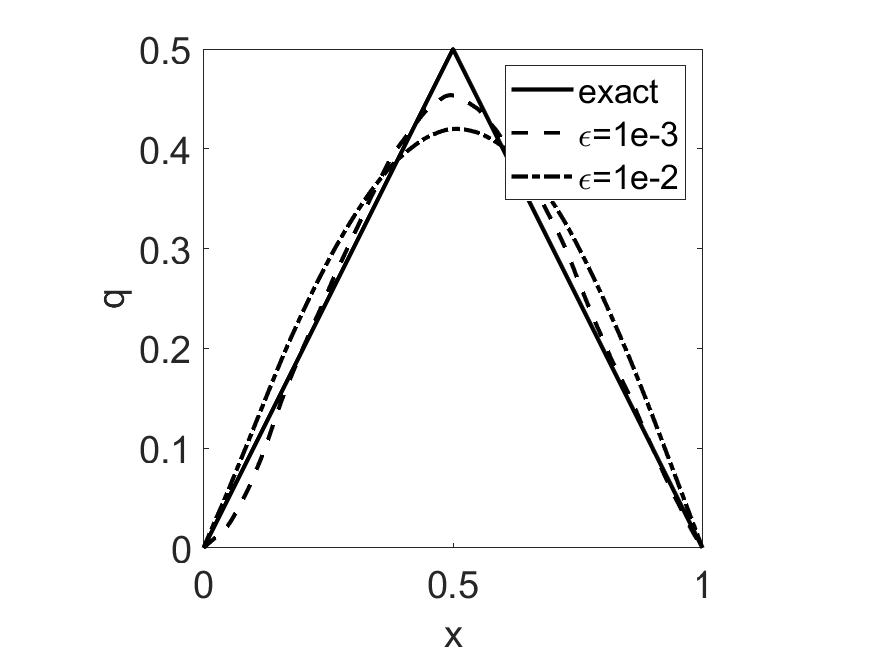} & \includegraphics[width=0.45\textwidth,trim={1.5cm 0 1cm 0.0cm}]{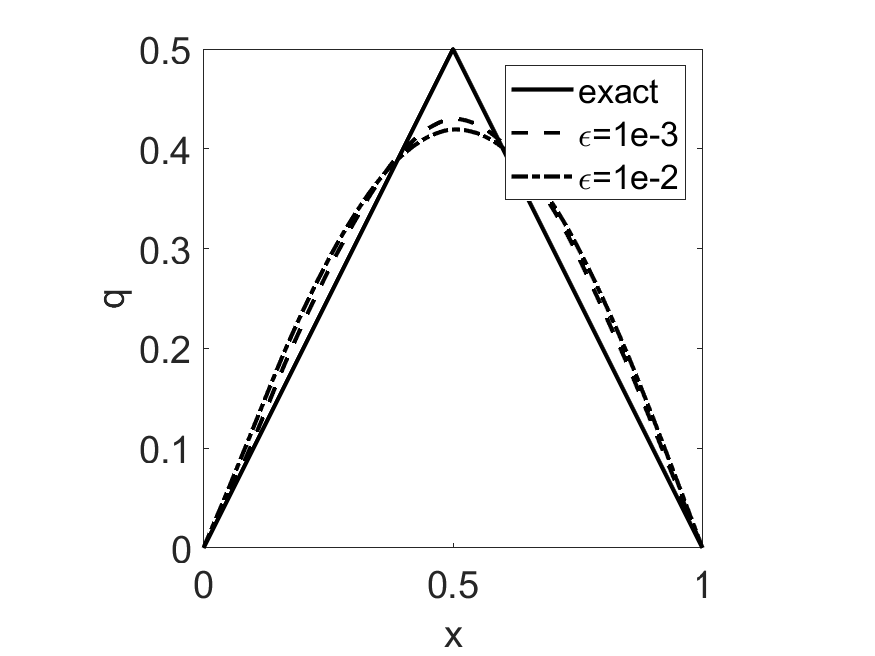}\\
  (a) $\alpha=0.25$ & (b) $\alpha=0.50$ \\
  \includegraphics[width=0.45\textwidth,trim={1.5cm 0 1cm 0.0cm}]{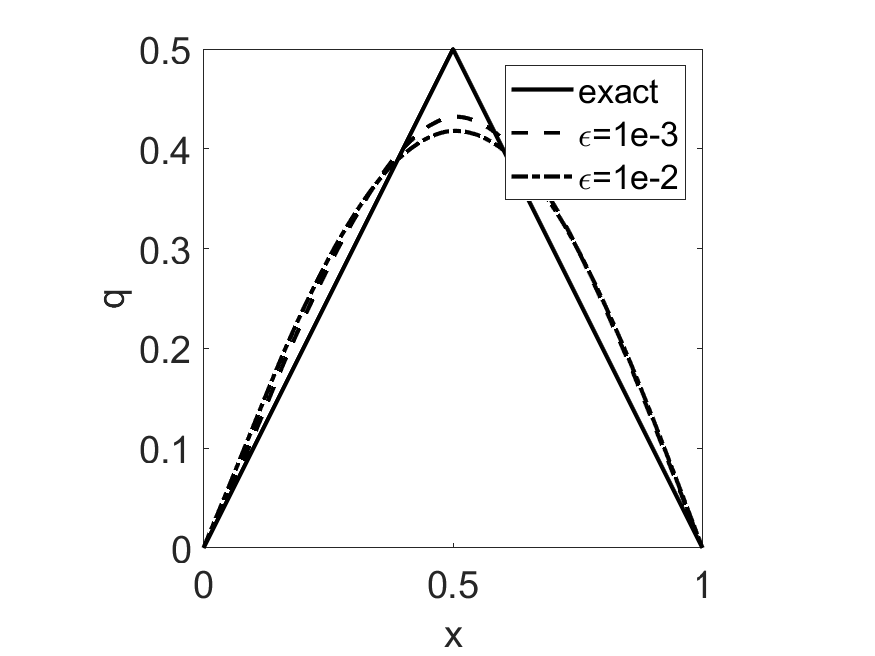} & \includegraphics[width=0.45\textwidth,trim={1.5cm 0 1cm 0.0cm}]{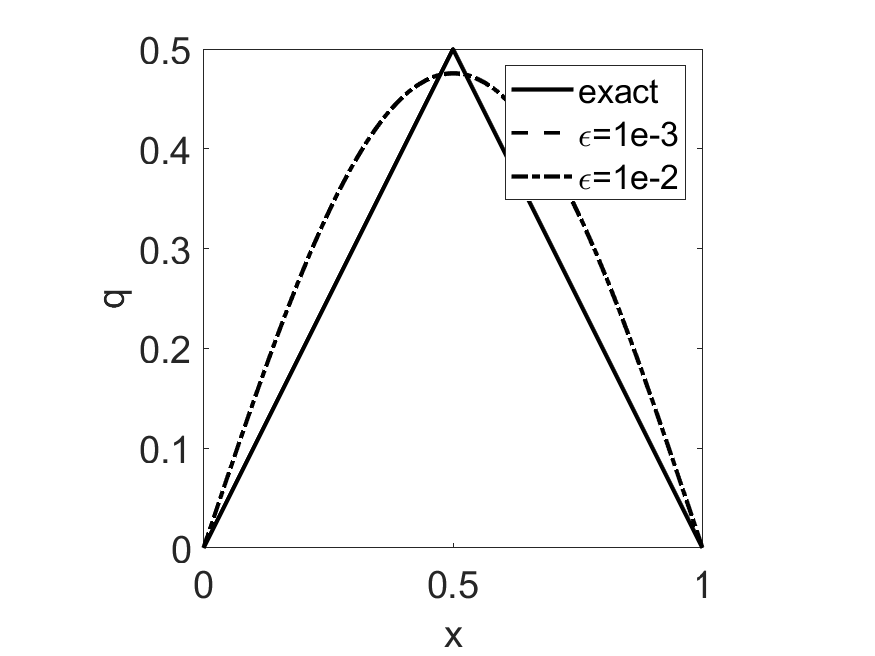}\\
  (c) $\alpha=0.75$ & (d) $\alpha=1.00$
  \end{tabular}
  \caption{Numerical reconstructions for Example \ref{exam:1d2} at $T=1$ with different $\alpha$ values.\label{fig:exam2:recon}}
\end{figure}

According to Tables \ref{tab:exam2T001aa}--\ref{tab:exam2T1aa}, the overall algorithm converges
within 5 iterations. In contrast, the convergence of the fixed point algorithm \eqref{eqn:fp}
requires many more iterations; see Table \ref{tab:exam2} for Example \ref{exam:1d2} with $T=1$.
Nonetheless, except for the case $\alpha=1.00$, the reconstruction errors are largely comparable. Thus, Anderson acceleration is
also effective in speeding up the convergence for recovering nonsmooth potentials. The plots of iterate convergence in Fig. \ref{fig:exam2:conv-aa}
again show a clear semi-convergence phenomenon. Note that overall the fixed point iterates is still
descent with respect to the residual, although there is one oscillation at the beginning. The
oscillation is related to the fact that the chosen $\lambda$ is fairly large.

\begin{table}[hbt!]
  \centering
  \caption{The reconstruction error $e_q$ for Example \ref{exam:1d2} with $T=1$, without Anderson acceleration.\label{tab:exam2}}
  \begin{tabular}{cccccc}
  \toprule
  $\alpha\backslash \epsilon$ & 0 & 1e-3 & 5e-3 & 1e-2 & 5e-2 \\
  \midrule
  0.25 & 6.85e-3 (1000) &  2.67e-2 (43) &  3.39e-2 (3) &  3.49e-2 (3) &  5.66e-2 (1)\\
  0.50 & 8.24e-3 (1000) &  2.82e-2 (51) &  3.43e-2 (5) &  3.46e-2 (4) &  4.43e-2 (2)\\
  0.75 & 1.16e-2 (1000) &  3.03e-2 (79) &  3.45e-2 (9) &  3.44e-2 (7) &  4.51e-2 (4)\\
  1.00 & 5.65e-2 (1000) &  5.66e-2 (1000)& 5.67e-2 (1000) & 5.70e-2 (1000) & 5.61e-2 (864)\\
  \bottomrule
  \end{tabular}
\end{table}

\begin{figure}[hbt!]
  \centering
  \begin{tabular}{cc}
    \includegraphics[width=0.45\textwidth,trim={1.5cm 0 1cm 0.5cm}]{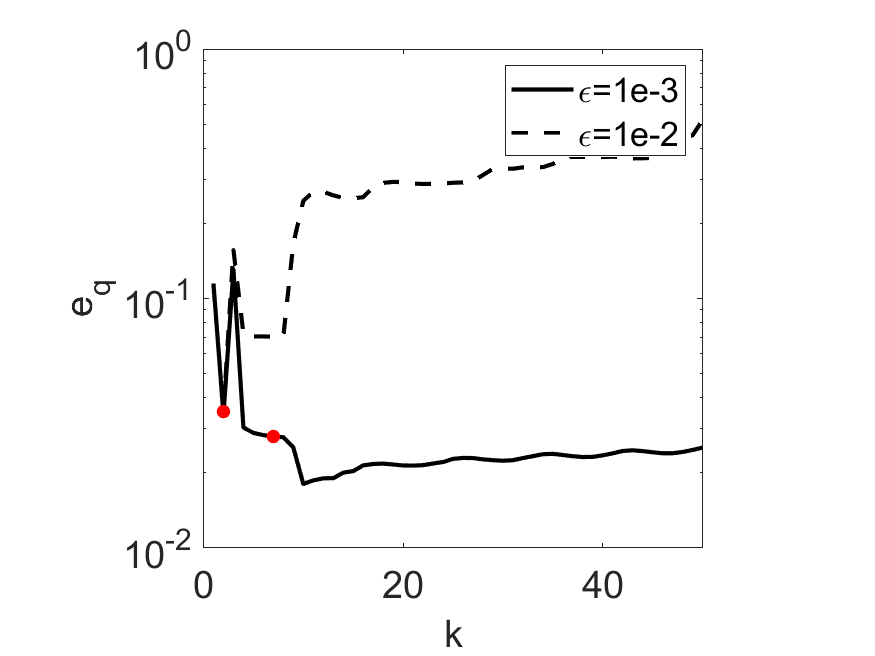} & \includegraphics[width=0.45\textwidth,trim={1.5cm 0 1cm 0.5cm}]{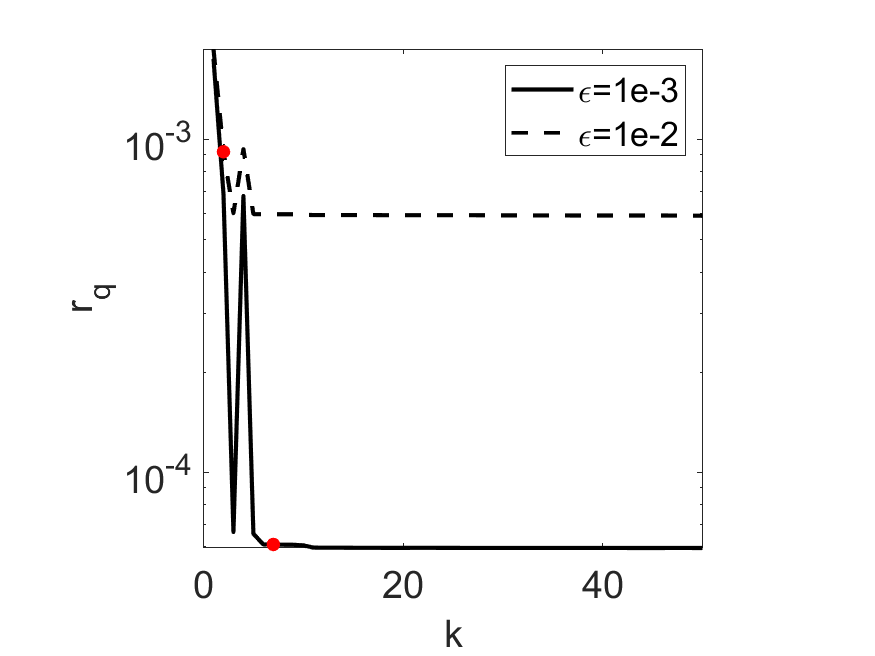}\\
    (a) error $e_q$ & (b) residual $r_q$
  \end{tabular}
  \caption{Convergence behavior of the fixed point algorithm with Anderson acceleration for Example \ref{exam:1d2} with $\alpha=0.5$.
  In the plots, the red circle refers to the stopping index determined by the discrepancy principle \eqref{eqn:fp}. \label{fig:exam2:conv-aa}}
\end{figure}

\subsection{Results for the two-dimensional case}
Last we given a two-dimensional example on the unit square $\Omega=(0,1)^2$ with a smooth coefficient. The domain
$\Omega$ is first partitioned smaller square of side length $1/M$, and then a uniform triangulation is obtained
by connecting the upper left and lower right vertices. The data is generated using $M=200$ and $N=1000$, and for the
inversion, the discretization parameters are taken to be $M=100$ and $N=500$. The fixed point algorithm \eqref{eqn:fp} is run for a maximum 200 iterations and
the relaxation parameter $\lambda$ is fixed at $100$, which is very conservative for scheme \eqref{eqn:fp}.

\begin{example}\label{exam:2d}
$u_0(x_1,x_2)=\sin^2\pi x_2$ and $q^\dag(x_1,x_2) = \sin(\pi x_1)x_2(1-x_2)$.
\end{example}

The initial condition $u_0$ does not satisfy the condition in Theorem \ref{thm:inverse}.
The numerical results for Example \ref{exam:2d} at $T=0.1$ are shown in Table \ref{tab:exam2daa} and Fig.
\ref{fig:exam2d}. With $\lambda=100$, the fixed point method \eqref{eqn:fp} can only converge
very slowly, and requires thousands of iterations to yield reasonable reconstruction, and thus
the corresponding results are not shown. Anderson
acceleration can greatly speed up the convergence, so that with any fixed noise level $\epsilon>0$, it converges
in two iterations. The method converges steadily, and the reconstruction error $e_q$ decreases
steadily as the noise level $\epsilon$ tends to zero. Up to $\epsilon$=1e-2 noise in the data, the result
represents an excellent reconstruction of the true potential $q^\dag$.

\begin{table}[hbt!]
  \centering
  \caption{The reconstruction error $e_q$ for Example \ref{exam:2d} at $T=0.1$.\label{tab:exam2daa}}
  \begin{tabular}{cccccc}
  \toprule
  $\alpha\backslash \epsilon$ & 0 & 1e-3 & 5e-3 & 1e-2 & 3e-2 \\
  \midrule
  0.25 & 2.29e-3 (200) &  6.27e-3 (2) &  6.91e-3 (2) &  8.16e-3 (2) &  1.52e-2 (2) \\
  0.50 & 5.59e-3 (200) &  6.11e-3 (2) &  6.23e-3 (2) &  6.94e-3 (2) &  1.30e-2 (2) \\
  0.75 & 1.13e-2 (200) &  8.03e-3 (2) &  7.55e-3 (2) &  7.25e-3 (2) &  9.48e-3 (2) \\
  1.00 & 1.37e-2 (200) &  1.19e-2 (5) &  1.03e-2 (2) &  9.83e-3 (2) &  8.95e-3 (2) \\
  \bottomrule
  \end{tabular}
\end{table}

\begin{figure}[hbt!]
  \centering
  \begin{tabular}{ccc}
    \includegraphics[width=0.3\textwidth,trim={1.5cm 0 1cm 0.5cm}]{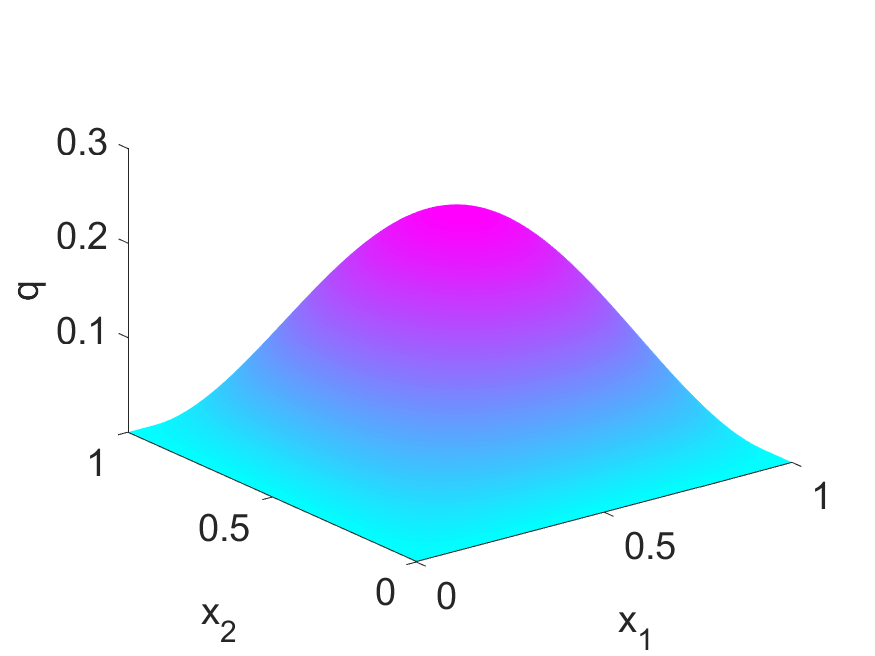} & \includegraphics[width=0.3\textwidth,trim={1.5cm 0 1cm 0.5cm}]{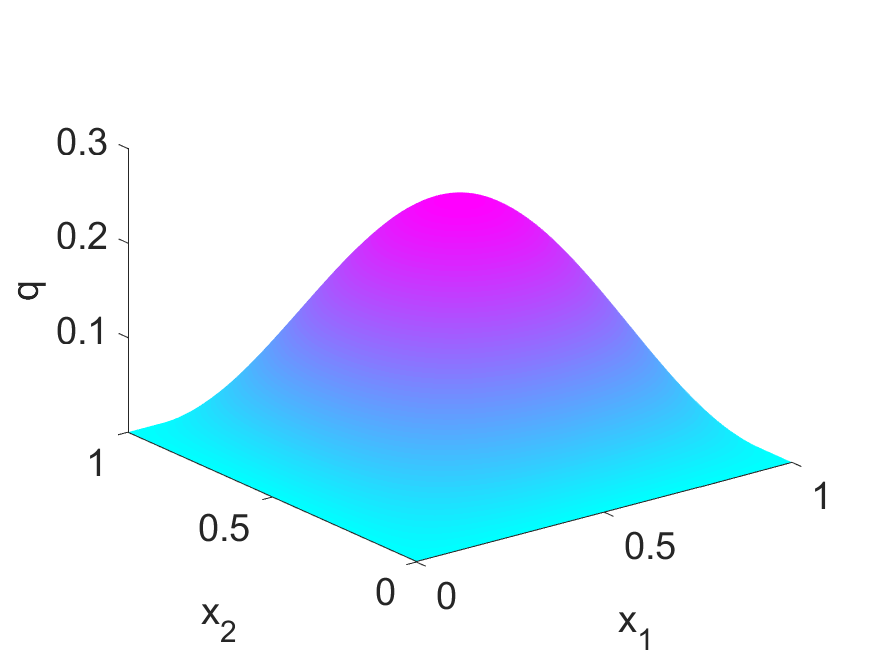}
    & \includegraphics[width=0.3\textwidth,trim={1.5cm 0 1cm 0.5cm}]{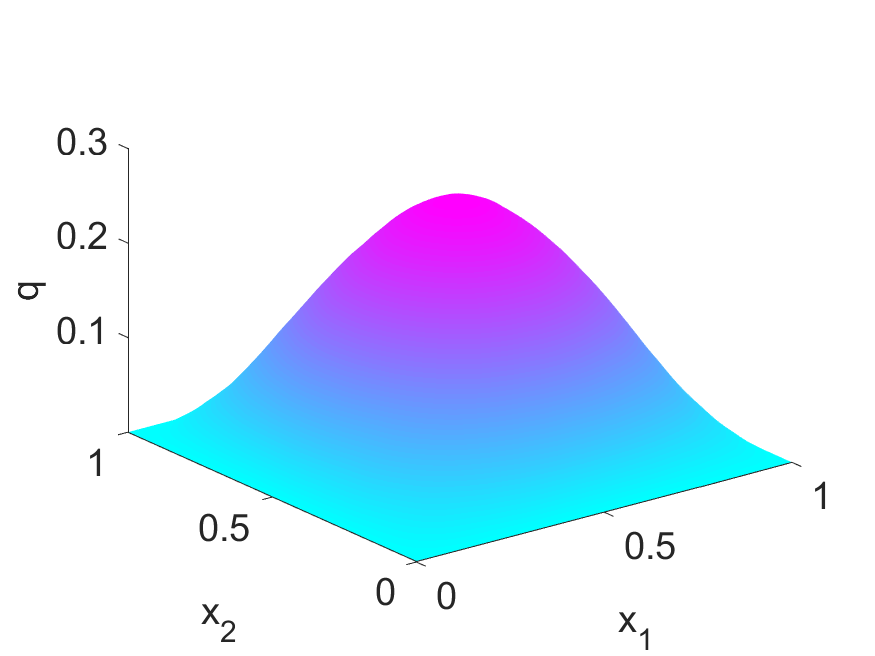} \\
    & \includegraphics[width=0.3\textwidth,trim={1.5cm 0 1cm 0.5cm}]{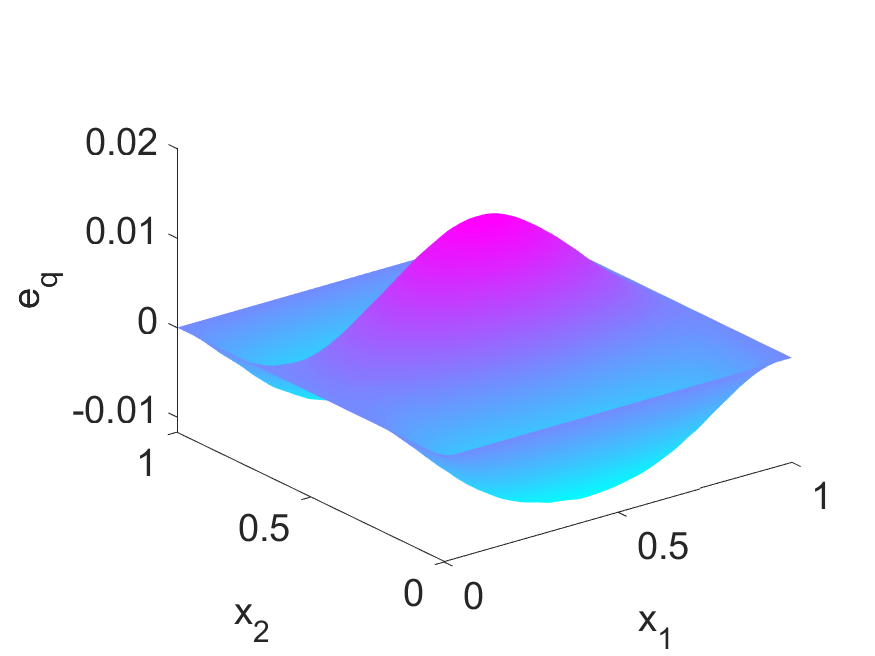} & \includegraphics[width=0.3\textwidth,trim={1.5cm 0 1cm 0.5cm}]{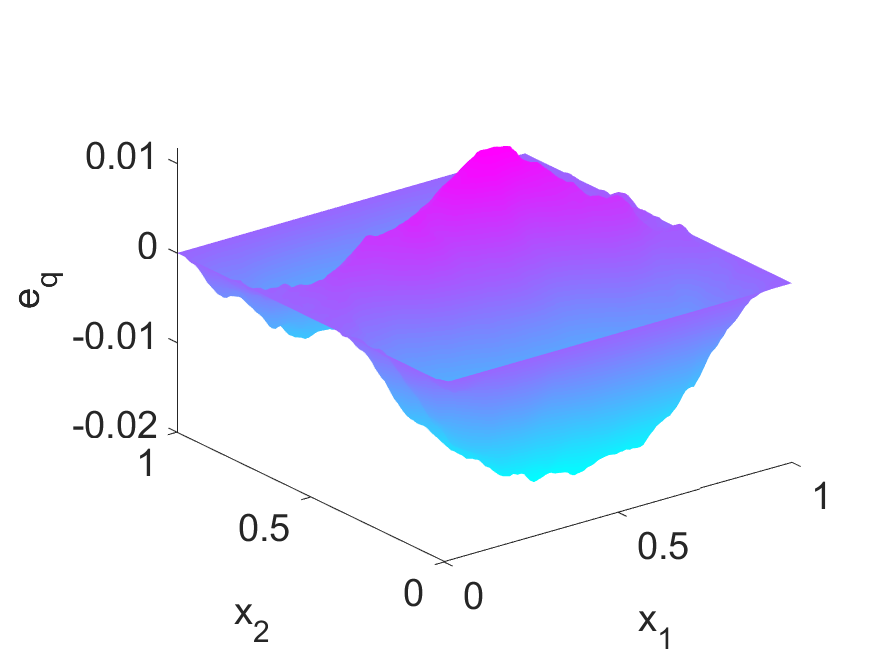} \\
   (a) exact $q^\dag $ & (b) $\epsilon$=1e-3 & (c) $\epsilon$=1e-2
  \end{tabular}
  \caption{Numerical results for Example \ref{exam:2d} with $T=0.1$ and $\alpha=0.5$: the first and second rows refer to the
  reconstruction $q$ and the pointwise error $e_q=q-q^\dag$. \label{fig:exam2d}}
\end{figure}

\section{Conclusion}

In this work, we have presented a study on the inverse problem of recovering a potential in
the subdiffusion model from terminal data. Under certain restrictions on the initial data, we
have derived a local Lipschitz stability result, using refined properties of Mittag-Leffler
functions. Further, we have developed a simple fixed point algorithm for recovering the potential
coefficient. When equipped with the discrepancy principle and Anderson acceleration, extensive
numerical experiments indicate that it is highly efficient and accurate.

There are a few interesting questions on the inverse potential problem awaiting answers.
First, the numerical experiments indicate a descent property of the fixed point iteration
for the residual, which however remains to be established in the general case. Second, it
is of much interest to analyze the regularizing property, e.g., convergence and convergence rates, of the fixed point algorithm
(and the accelerated variant) when equipped with the discrepancy principle. Third, it is natural to ask whether it is
possible to recover the potential and the fractional order $\alpha$ simultaneously from the
terminal data, and if so, also to derive relevant stability estimates. In the case of lateral
Cauchy data, it is known that one can recover the diffusion coefficient and fractional order together
\cite{ChengNakagawaYamamotoYamazaki:2009}. We shall explore these issues in future works.

\section*{Acknowledgements}
The authors are grateful to the anonymous referees and the board member for several
constructive comments, which have led to an improvement in the quality of
the paper.

\bibliographystyle{abbrv}
\bibliography{frac}

\begin{thebibliography}{10}

\bibitem{AdamsFournier:2003}
R.~A. Adams and J.~J.~F. Fournier.
\newblock {\em Sobolev {S}paces}.
\newblock Elsevier/Academic Press, Amsterdam, 2nd edition, 2003.

\bibitem{Anderson:1965}
D.~G. Anderson.
\newblock Iterative procedures for nonlinear integral equations.
\newblock {\em J. Assoc. Comput. Mach.}, 12:547--560, 1965.

\bibitem{BrezinskiSaad:2018}
C.~Brezinski, M.~Redivo-Zaglia, and Y.~Saad.
\newblock Shanks sequence transformations and {A}nderson acceleration.
\newblock {\em SIAM Rev.}, 60(3):646--669, 2018.

\bibitem{ChengNakagawaYamamotoYamazaki:2009}
J.~Cheng, J.~Nakagawa, M.~Yamamoto, and T.~Yamazaki.
\newblock Uniqueness in an inverse problem for a one-dimensional fractional
  diffusion equation.
\newblock {\em Inverse Problems}, 25(11):115002, 16, 2009.

\bibitem{ChoulliYamamoto:1996}
M.~Choulli and M.~Yamamoto.
\newblock Generic well-posedness of an inverse parabolic problem---the
  {H}\"{o}lder-space approach.
\newblock {\em Inverse Problems}, 12(3):195--205, 1996.

\bibitem{ChoulliYamamoto:1997}
M.~Choulli and M.~Yamamoto.
\newblock An inverse parabolic problem with non-zero initial condition.
\newblock {\em Inverse Problems}, 13(1):19--27, 1997.

\bibitem{Ciarlet:2013}
P.~G. Ciarlet.
\newblock {\em Linear and {N}onlinear {F}unctional {A}nalysis with
  {A}pplications}.
\newblock SIAM, Philadelphia, PA, 2013.

\bibitem{CourantHilbert:1953}
R.~Courant and D.~Hilbert.
\newblock {\em Methods of {M}athematical {P}hysics}, volume~I.
\newblock Interscience, New York, 1953.

\bibitem{EnglHankeNeubauer:1996}
H.~W. Engl, M.~Hanke, and A.~Neubauer.
\newblock {\em Regularization of {I}nverse {P}roblems}.
\newblock Kluwer, Dordrecht, 1996.

\bibitem{EvansPollock:2020}
C.~Evans, S.~Pollock, L.~G. Rebholz, and M.~Xiao.
\newblock A proof that {A}nderson acceleration improves the convergence rate in
  linearly converging fixed-point methods (but not in those converging
  quadratically).
\newblock {\em SIAM J. Numer. Anal.}, 58(1):788--810, 2020.

\bibitem{HatanoHatano:1998}
Y.~Hatano and N.~Hatano.
\newblock Dispersive transport of ions in column experiments: An explanation of
  long-tailed profiles.
\newblock {\em Water Res. Research}, 34(5):1027--1033, 1998.

\bibitem{HenryLanglandsWearne:2006}
B.~I. Henry, T.~A.~M. Langlands, and S.~L. Wearne.
\newblock Anomalous diffusion with linear reaction dynamics: From continuous
  time random walks to fractional reaction-diffusion equations.
\newblock {\em Phys. Rev. E}, 74:031116, 2006.

\bibitem{Isakov:1991}
V.~Isakov.
\newblock Inverse parabolic problems with the final overdetermination.
\newblock {\em Comm. Pure Appl. Math.}, 44(2):185--209, 1991.

\bibitem{ItoJin:2015}
K.~Ito and B.~Jin.
\newblock {\em Inverse {P}roblems: {T}ikhonov {T}heory and {A}lgorithms}.
\newblock World Scientific Publishing Co. Pte. Ltd., Hackensack, NJ, 2015.

\bibitem{JinLazarovZhou:2016}
B.~Jin, R.~Lazarov, and Z.~Zhou.
\newblock Two fully discrete schemes for fractional diffusion and
  diffusion-wave equations with nonsmooth data.
\newblock {\em SIAM J. Sci. Comput.}, 38(1):A146--A170, 2016.

\bibitem{JinLazarovZhou:2019}
B.~Jin, R.~Lazarov, and Z.~Zhou.
\newblock Numerical methods for time-fractional evolution equations with
  nonsmooth data: a concise overview.
\newblock {\em Comput. Methods Appl. Mech. Engrg.}, 346:332--358, 2019.

\bibitem{JinLiZhou:2018}
B.~Jin, B.~Li, and Z.~Zhou.
\newblock Numerical analysis of nonlinear subdiffusion equations.
\newblock {\em SIAM J. Numer. Anal.}, 56(1):1--23, 2018.

\bibitem{JinRundell:2012}
B.~Jin and W.~Rundell.
\newblock An inverse problem for a one-dimensional time-fractional diffusion
  problem.
\newblock {\em Inverse Problems}, 28(7):075010, 19, 2012.

\bibitem{JinRundell:2015}
B.~Jin and W.~Rundell.
\newblock A tutorial on inverse problems for anomalous diffusion processes.
\newblock {\em Inverse Problems}, 31(3):035003, 40, 2015.

\bibitem{KaltenbacherRundell:2019}
B.~Kaltenbacher and W.~Rundell.
\newblock On an inverse potential problem for a fractional reaction-diffusion
  equation.
\newblock {\em Inverse Problems}, 35(6):065004, 31, 2019.

\bibitem{KianYamamoto:2019}
Y.~Kian and M.~Yamamoto.
\newblock Reconstruction and stable recovery of source terms and coefficients
  appearing in diffusion equations.
\newblock {\em Inverse Problems}, 35(11):115006, 24, 2019.

\bibitem{KilbasSrivastavaTrujillo:2006}
A.~A. Kilbas, H.~M. Srivastava, and J.~J. Trujillo.
\newblock {\em Theory and {A}pplications of {F}ractional {D}ifferential
  {E}quations}, volume 204 of {\em North-Holland Mathematics Studies}.
\newblock Elsevier Science B.V., Amsterdam, 2006.

\bibitem{KlibanovLiZhang:2020}
M.~V. Klibanov, J.~Li, and W.~Zhang.
\newblock Convexification for an inverse parabolic problem.
\newblock Preprint, arXiv:2001.01880, 2020.

\bibitem{Krasnoschok:2016}
M.~V. Krasnoschok.
\newblock Solvability in {H}\"{o}lder space of an initial boundary value
  problem for the time-fractional diffusion equation.
\newblock {\em Zh. Mat. Fiz. Anal. Geom.}, 12(1):48--77, 2016.

\bibitem{LiChengLiu:2020}
Z.~Li, X.~Cheng, and Y.~Liu.
\newblock Generic well-posedness for an inverse source problem for a multi-term
  time-fractional diffusion equation.
\newblock {\em Taiwanese J. Math.}, 24(4):1005--1020, 2020.

\bibitem{LiuLiYamamoto:2019isp}
Y.~Liu, Z.~Li, and M.~Yamamoto.
\newblock Inverse problems of determining sources of the fractional partial
  differential equations.
\newblock In {\em Handbook of {F}ractional {C}alculus with {A}pplications.
  {V}ol. 2}, pages 411--429. De Gruyter, Berlin, 2019.

\bibitem{YikanRundellYamamoto:2016}
Y.~Liu, W.~Rundell, and M.~Yamamoto.
\newblock Strong maximum principle for fractional diffusion equations and an
  application to an inverse source problem.
\newblock {\em Fract. Calc. Appl. Anal.}, 19(4):888--906, 2016.

\bibitem{LuchkoYamamoto:2017}
Y.~Luchko and M.~Yamamoto.
\newblock On the maximum principle for a time-fractional diffusion equation.
\newblock {\em Fract. Calc. Appl. Anal.}, 20(5):1131--1145, 2017.

\bibitem{MetzlerKlafter:2000}
R.~Metzler and J.~Klafter.
\newblock The random walk's guide to anomalous diffusion: a fractional dynamics
  approach.
\newblock {\em Phys. Rep.}, 339(1):77, 2000.

\bibitem{MillerYamamoto:2013}
L.~Miller and M.~Yamamoto.
\newblock Coefficient inverse problem for a fractional diffusion equation.
\newblock {\em Inverse Problems}, 29(7):075013, 8, 2013.

\bibitem{Nigmatulin:1986}
R.~R. Nigmatulin.
\newblock The realization of the generalized transfer equation in a medium with
  fractal geometry.
\newblock {\em Phys. Stat. Sol. B}, 133:425--430, 1986.

\bibitem{Pollard:1948}
H.~Pollard.
\newblock The completely monotonic character of the {M}ittag-{L}effler function
  {$E_a(-x)$}.
\newblock {\em Bull. Amer. Math. Soc.}, 54:1115--1116, 1948.

\bibitem{PrilepkoSolovev:1987}
A.~I. Prilepko and V.~V. Solov$'$ev.
\newblock On the solvability of inverse boundary value problems for the
  determination of the coefficient preceding the lower derivative in a
  parabolic equation.
\newblock {\em Differentsial$'$ nye Uravneniya}, 23(1):136--143, 182, 1987.

\bibitem{Rundell:1987}
W.~Rundell.
\newblock The determination of a parabolic equation from initial and final
  data.
\newblock {\em Proc. Amer. Math. Soc.}, 99(4):637--642, 1987.

\bibitem{SakamotoYamamoto:2011}
K.~Sakamoto and M.~Yamamoto.
\newblock Initial value/boundary value problems for fractional diffusion-wave
  equations and applications to some inverse problems.
\newblock {\em J. Math. Anal. Appl.}, 382(1):426--447, 2011.

\bibitem{SakamotoYamamoto:2011isp}
K.~Sakamoto and M.~Yamamoto.
\newblock Inverse source problem with a final overdetermination for a
  fractional diffusion equation.
\newblock {\em Math. Control Relat. Fields}, 1(4):509--518, 2011.

\bibitem{SeyboldHilfer:2008}
H.~Seybold and R.~Hilfer.
\newblock Numerical algorithm for calculating the generalized
  {M}ittag-{L}effler function.
\newblock {\em SIAM J. Numer. Anal.}, 47(1):69--88, 2008/09.

\bibitem{Simon:2014}
T.~Simon.
\newblock Comparing {F}r\'{e}chet and positive stable laws.
\newblock {\em Electron. J. Probab.}, 19:no. 16, 25, 2014.

\bibitem{TothKelley:2015}
A.~Toth and C.~T. Kelley.
\newblock Convergence analysis for {A}nderson acceleration.
\newblock {\em SIAM J. Numer. Anal.}, 53(2):805--819, 2015.

\bibitem{WalkerNi:2011}
H.~F. Walker and P.~Ni.
\newblock Anderson acceleration for fixed-point iterations.
\newblock {\em SIAM J. Numer. Anal.}, 49(4):1715--1735, 2011.

\bibitem{YusteAbad:2010}
S.~B. Yuste, E.~Abad, and K.~Lindenberg.
\newblock Reaction-subdiffusion model of morphogen gradient formation.
\newblock {\em Phys. Rev. E}, 82:061123, 2010.

\bibitem{ZhangZhou:2017}
Z.~Zhang and Z.~Zhou.
\newblock Recovering the potential term in a fractional diffusion equation.
\newblock {\em IMA J. Appl. Math.}, 82(3):579--600, 2017.

\end{thebibliography}
\end{document}